\documentclass[11pt,reqno]{amsart}

\usepackage[all]{xy}
\usepackage{amscd}
\usepackage{mathrsfs}
\usepackage{amssymb}
\usepackage{amsthm}
\usepackage{amsmath}
\usepackage{enumerate}
\usepackage{indentfirst}
\usepackage{dsfont}

\usepackage{xcolor}
\usepackage{comment}

\usepackage{hyperref}

\newtheorem{thm}{Theorem}[section]
\newtheorem{conj}[thm]{Conjecture}
\newtheorem{cor}[thm]{Corollary}
\newtheorem{lem}[thm]{Lemma}
\newtheorem{prop}[thm]{Proposition}

\theoremstyle{definition}
\newtheorem{defn}[thm]{Definition}

\theoremstyle{remark}
\newtheorem{rem}[thm]{Remark}

\numberwithin{equation}{section}

\newtheorem{definition}[thm]{Definition}

\newtheorem{remark}[thm]{Remark}

\newtheorem{numbering}[thm]{}
\newtheorem{defrmk}[thm]{Definition and Remark}

\newcommand{\CaA}{\mathcal A}
\newcommand{\CaB}{\mathcal B}

\newcommand{\CaJ}{\mathcal J}

\newcommand{\frg}{\mathfrak g}

\newcommand{\bbR}{\mathbb R}

\newcommand{\bbZ}{\mathbb Z}

\newcommand{\rma}{a}

\newcommand{\xo}{x}
\newcommand{\yo}{y}

\newcommand{\abrank}{r}

\newcommand{\sfK}{\mathsf K}

\newcommand{\val}{\nu}
\newcommand{\Ad}{\mathrm{Ad}}

\newcommand{\pid}{\mathfrak p}

\newcommand{\Tr}{\mathrm{Tr}}

\newcommand{\bG}{\mathbf G}

\newcommand{\bC}{\mathbf C}
\newcommand{\bT}{\mathbf T}
\newcommand{\bL}{\mathbf L}
\newcommand{\bM}{\mathbf M}

\newcommand{\bU}{\mathbf U}

\newcommand{\bK}{\sfK}
\newcommand{\bZ}{\mathbf Z}

\newcommand{\rK}{K}
\newcommand{\rJ}{J}

\newcommand{\lieG}{\mathfrak g}

\newcommand{\lieT}{\mathfrak t}

\newcommand{\blieG}{\boldsymbol{\lieG}}

\newcommand{\Xs}{X}

\newcommand{\gsc}{\mathcal S}

\newcommand{\datum}{\Sigma}

\newcommand{\rs}{\mathrm{reg}}

\newcommand{\Bd}{\CaB}

\newcommand{\Apt}{\CaA}

\newcommand{\dpi}{\varrho}
\newcommand{\rtm}{r}
\newcommand{\const}{C}
\newcommand{\urtm}{\rtm_\circ}

\newcommand{\stm}{s}

\newcommand{\supp}{\mathrm{Supp}}
\newcommand{\reg}{{\mathrm{reg}}}
\newcommand{\nil}{\mathcal N}

\newcommand{\Hom}{\mathrm{Hom}}
\newcommand{\Ind}{\mathrm{Ind}}
\newcommand{\cind}{\textrm{c-}\mathrm{ind}}
\newcommand{\ch}{\mathrm{ch}}
\newcommand{\Gal}{\mathrm{Gal}}

\newcommand{\vol}{\mathrm{vol}}

\newcommand{\Hypk}{\textup{(H$k$)}}
\newcommand{\HypB}{\textup{(HB)}}
\newcommand{\HypGT}{\textup{(HGT)}}
\newcommand{\HypN}{\textup{(H$\nil$)}}

\newcommand{\midvsp}{\vspace{7pt}}

\newcommand{\cB}{{\mathcal{B}}}

\newcommand{\fkg}{{\mathfrak g}}
\newcommand{\fkh}{{\mathfrak h}}

\newcommand{\Z}{\mathbb{Z}}

\newcommand{\Q}{\mathbb{Q}}

\newcommand{\R}{\mathbb{R}}

\newcommand{\C}{\mathbb{C}}

\newcommand{\scusp}{\mathrm{sc}}

\newcommand{\hg}{\mathrm{ht}}
\newcommand{\aconst}{A}

\newcommand{\bs}{\backslash}

\def\Gal{{\rm Gal}}

\newcommand{\ind}{{\rm Ind}}

\newcommand{\sd}{\mathrm{sd}}
\newcommand{\fdeg}{\mathrm{deg}}

\newcommand{\sgn}{{\rm sgn}}
\newcommand{\tr}{{\rm tr}\,}

\newcommand{\End}{{\rm End}}

\newcommand{\GL}{{\rm GL}}
\newcommand{\Lie}{{\rm Lie}\,}

\newcommand{\Yu}{{\rm Yu}}

\newcommand{\cO}{\mathcal{O}}

\newcommand{\cX}{\mathcal{X}}

\newcommand{\Irr}{{\rm Irr}}

\newcommand{\Fix}{{\rm Fix}}

\def\hat{\widehat}

\def\ra{\rightarrow}

\def\ol{\overline}
\def\tilde{\widetilde}

\newcommand{\SL}{\mathrm{SL}}

\def\benu{\begin{enumerate}}
\def\eenu{\end{enumerate}}
\def\beq{\begin{equation}}
\def\eeq{\end{equation}}
\def\bit{\begin{itemize}}
\def\eit{\end{itemize}}
\DeclareMathOperator {\Mtr}  {tr}
\providecommand{\set}[1]{\left\{#1\right\}}
\providecommand{\mdede}[4]{
\left(
\begin{smallmatrix}
#1&#2  \\
#3&#4 
\end{smallmatrix}
\right)
}
\newcommand{\Mvol} {vol}
\providecommand{\abs}[1]{\ensuremath{\left|#1\right|}}
\providecommand{\Lquote}[1]{``#1"}

\begin{document}

\title[Local Constancy]{Asymptotics and local constancy of characters of $p$-adic groups}

\author{Ju-Lee Kim}\email{julee@math.mit.edu}
\address{Department of Mathematics, Massachusetts Institute of Technology,
77 Massachusetts Avenue, Cambridge, MA 02139, USA}
\author{Sug Woo Shin}\email{sug.woo.shin@berkeley.edu}
\address{Department of Mathematics, UC Berkeley, Berkeley, CA 94720, USA$//$ Korea Institute for Advanced Study, 85 Hoegiro,
Dongdaemun-gu, Seoul 130-722, Republic of Korea}\thanks{J.-L. Kim was partially supported by NSF grants, S. W. S. was partially supported by NSF grant DMS-1162250/1449558
  and a Sloan Fellowship. N. T. was partially supported by NSF grant DMS-1200684/1454893.}
\author{Nicolas Templier}\email{templier@math.cornell.edu}
\address{Department of Mathematics, Cornell University,
Ithaca, NY 14853, USA}

\date{\today}

\begin{abstract}
In this paper we study quantitative aspects of trace characters $\Theta_\pi$ of reductive $p$-adic groups when the representation $\pi$ varies. Our approach is based on the local constancy of characters and we survey some other related results.  We formulate a conjecture on the behavior of $\Theta_\pi$ relative to the formal degree of $\pi$, which  we are able to prove in the case where $\pi$ is a tame supercuspidal. The proof builds on J.-K.~Yu's construction and the structure of Moy-Prasad subgroups.
\end{abstract}

\maketitle

\tableofcontents

\section{Introduction}\label{s:intro}
For an admissible representation $\pi$ of a $p$-adic reductive group $G$, its trace character distribution is defined by
\[
\langle \Theta_\pi, f \rangle = \Mtr \pi(f),\quad f\in \mathcal{C}_c(G).
\]
 Harish-Chandra showed that it is represented by a locally integrable function on $G$ still denoted by $\Theta_\pi$, which moreover is locally constant on the open subset of regular elements.

Our goal in this paper is to initiate a quantitative theory of trace characters $\Theta_\pi$ when the representation $\pi$ varies. One motivation is towards a better understanding of the spectral side of the trace formula where one would like to control the global behavior of characters~\cite{KST}. Another motivation comes from the Weyl character formula. For a finite dimensional representation $\sigma$ of a compact Lie group and a regular element $\gamma$,
\[
 D(\gamma)^{\frac12} |\tr \sigma(\gamma)|
\le
|W|,
\]
where $W$ is the Weyl group and $D(\gamma)$ is the Weyl discriminant which appears in the denominator of the character formula. More generally the Harish-Chandra formula for characters of discrete series yields similar estimates for real reductive groups, see \S\ref{sub:real-group-char} below.

If $\pi$ is a square-integrable representation of $G$ we denote by $\deg(\pi)$ its formal degree. Let $\gamma$ be a fixed regular semisimple element. The central conjecture we would like to propose in this paper (Conjecture \ref{c:asymptotic-char}) is essentially that $\frac{\Theta_\pi(\gamma)}{\deg(\pi)}$ converges to zero as $\deg(\pi)$ grows.

It is nowadays possible to study such a question thanks to recent progress in constructing supercuspidal representations and computing their trace characters, see notably~\cite{ADSS:supercuspidal-characters,Adler-Spice:supercuspidal-characters} and the references there.

The main result of this paper (Theorems \ref{t:asymptotic-char} and \ref{thm:uniform-bound}, with the latter improved as in \S\ref{sub:assumption} below) verifies our conjecture for the tame supercuspidal representations $\pi$ constructed by J.-K.~Yu for topologically unipotent elements $\gamma$ when the residual characteristic of the base field is large enough (in an effective manner). In such a setup we establish that for some constants $A,\kappa>0$ depending on the group $G$,
\begin{equation}\label{mainresult}
\frac{D(\gamma)^A  |\Theta_\pi(\gamma)|}{\deg(\pi)^{1-\kappa}}
\end{equation}
is bounded both as a function of $\gamma$ topologically unipotent and as $\pi$ varies over the set of irreducible supercuspidal representations of $G$.

Yu's contruction gives tame supercuspidal representations $\pi=\cind_J^G \rho$ as compactly induced from an explicit open compact-modulo-center subgroup $J$ given in terms of a sequence of tamely ramified twisted Levi subgroups (whose definition is recalled in \S\ref{sub:generic-datum} below).  The main theorem of Yu~\cite{Yu01} is that the induction is irreducible, and therefore is supercuspidal. This may be summarized by the inclusions,
\[
	\Irr^{\mathrm{Yu}}(G) \subset \Irr^{\cind}(G) \subset \Irr^{\mathrm{sc}}(G),
\]
where $\Irr^{\mathrm{sc}}(G)$ consists of all irreducible supercuspidal representations (up to isomorphism), and the first two subsets are given by Yu's construction and by compact induction from open compact-modulo-center subgroups, respectively.
The formal degree $\deg(\pi)$ is proportional to $\dim(\rho)/\vol(J)$. Moreover the first-named author~\cite{Kim07} has shown that if the residue characteristic is large enough, then Yu's construction exhausts all supercuspidals, i.e. the above inclusions are equalities. This means that our result \eqref{mainresult} is true for \emph{all} supercuspidal representations in that case.

One important ingredient in proving our main result is using the local constancy of characters. For a given regular semisimple element $\gamma$, if $\Theta_\pi$ is constant on $\gamma K$ for a (small) open compact subgroup $K$ of $G$, then
\begin{equation}\label{i:locconst}
\Theta_\pi(\gamma)=\frac1{\vol(K)}
\langle
\Theta_\pi , 1_{\gamma K}
\rangle
=\mathrm{trace}(\pi(\gamma)|V_\pi^{K})
\end{equation}
where $1_{\gamma K}$ is the characteristic function of $\gamma K$. The results of \cite{AK07} and \cite{Meyer-Solleveld:growth} determine the size of $K$, which depends on the (Moy-Prasad) depth of $\pi$ and the singular depth of $\gamma$ (see Definition \ref{d:singular-depth} below). For our main result, as we vary $\pi$ such that the formal degree of $\pi$ increases (equivalently, the depth of $\pi$ increases), we choose $K$ appropriately to be able to approximate the size of $\Theta_\pi(\gamma)$. Write $G_x$ for the parahoric subgroup of $G$ associated to $x$. The fact that $\pi=\cind^G_J\rho$, via Mackey's formula, allows us to bound $|\Theta_\pi(\gamma)|/\deg(\pi)$ in terms of the number of fixed points of $\gamma$ (which may be assumed to lie in $G_{x}$) acting on $(G_{x} \cap gJg^{-1})\bs G_{x}$ by right translation for various $g\in G$. To bound the cardinality of the fixed points we prove quite a few numerical inequalities as Yu's data vary by a systematic study of Moy-Prasad subgroups in Yu's construction.

The celebrated regularity theorem of Harish-Chandra~\cite{HC70} says that $D(\gamma)^{\frac12} \Theta_\pi(\gamma)$ is locally bounded as a function of $\gamma$ and similarly for any $G$-invariant admissible distribution. It implies that $\Theta_\pi$ is given by a locally integrable function on $G$ and moreover there is a germ expansion~\cite{HC99} when $\gamma$ approaches a non-regular element. In comparison our result concerning~\eqref{mainresult} is much less precise but at the same time we also allow $\pi$ to vary.

The local constancy~\eqref{i:locconst} is used similarly in~\cite{KL:unipotent,KL:Steinberg} to compute the characters of unipotent representations at very regular elements. In such situation the depth of $\pi$ is sufficiently larger than the singular depth of $\gamma$, and the size of $K$ is determined by the depth of $\pi$. Another application of the local constancy of trace characters is~\cite{Meyer-Solleveld:growth} which considers trace characters of representations $\pi$ in positive characteristic different from $p$. Among other results they show that the trace character $\Theta_\pi$ exists as a function essentially as consequence of the formula~\eqref{i:locconst}.

\subsection*{Acknowledgment}
We would like to thank Julius Witte for his comments and also for pointing out a mistake in Lemma~\ref{lem: centralizer2} in an earlier version. We would like to thank Anne-Marie Aubert and the referee for their helpful comments.

\subsection*{Notation and Conventions}
Let $p$ be a prime. Let $k$ be a finite extension of $\Q_p$. Denote by $q$ the cardinality of the residue field of $k$. For any tamely ramified finite extension $E$ of $k$, let $\nu$ denote the valuation on $E$ which coincides with the valuation of $\Q_p$ when restricted. Let $\mathcal O_E$ and $\mathfrak p_E$ be the ring of integers in $E$ and the prime ideal of $\mathcal O_E$ respectively. We fix an additive character $\Omega_k$ of $k$ with conductor $\mathfrak p_k$.

  Let $\bG$ be a connected reductive group over $k$, whose Lie algebra is denoted $\blieG$. Let $\abrank_G$ be the difference between the absolute rank of $\bG$ (the dimension of any maximal torus in $\bG$) and the dimension of the center $\bZ_\bG$ of $G$. Write $G$ and $\lieG$ for $\bG(k)$ and $\blieG(k)$, respectively. The linear dual of $\lieG$ is denoted by $\lieG^*$. Denote the set of regular semisimple elements in $G$ by $G_{\rs}$.

Throughout the paper,  by a unipotent subgroup, we mean the unipotent subgroup given by the unipotent radical of a parabolic subgroup.

For a subset $S$ of a group $H$ and an element $g\in H$, we write $S^g$ or $\,^{g^{-1}}\!S$ for $g^{-1}Sg$. Similarly if $g,h\in H$, we write $h^g$ or $\,^{g^{-1}}\!h$ for $g^{-1}hg$. If $S$ is a subgroup of $H$ and $\xi$ is a representation of $S$, denote by $\xi^g$ or $\,^{g^{-1}}\!\xi$ the representation of $S^g=\,^{g^{-1}}\!S$ given by $\xi^g(s)=\,^{g^{-1}}\!\xi(s)=\xi(gs g^{-1})$, $s\in S$.

\section{Minimal K-types and Yu's construction of supercuspidal representations}\label{s:Yu-construction}

In this section we review the construction of supercuspidal representations of a $p$-adic reductive group from the so-called generic data due to Jiu-Kang Yu and recall a result by the first author that his construction exhausts all supercuspidal representations provided the residue characteristic of the base field is sufficiently large. The construction yields a supercuspidal representation concretely as a compactly induced representation, and this will be an important input in the next section.

\subsection{Moy-Prasad filtrations}\label{sub:notation Bd}

For a tamely ramified extension $E$ of $k$,
denote by $\Bd(\bG,E)$ (resp. $\Bd^{\mathrm{red}}(G)$) be the extended (resp. reduced) building of $\bG$ over $E$. When $E=k$, we write $\Bd(G)$ (resp. $\Bd^{\mathrm{red}}(G)$) for $\Bd(\bG,k)$ (resp. $\Bd^{\mathrm{red}}(\bG,k)$) for simplicity.
If $\bT$ is a maximal $E$-split $k$-torus, let
$\Apt(\bT,\bG,E)$ denote the apartment associated to $\bT$ in $\Bd(\bG,E)$.
When $E=k$, write $\Apt(T)$ for the same apartment.
It is known that for any tamely ramified
 Galois extension $E'$ of $E$,
$\Apt(\bT,\bG,E)$ can be identified with
the set of all $\Gal(E'/E)$-fixed points in
$\Apt(\bT,\bG,E')$. Likewise, $\Bd(\bG,E)$ can be embedded into $\Bd(\bG,E')$
and its image is equal to the set of the Galois fixed points in $\Bd(\bG,E')$
\cite{Rou77, Pra01}.

For $(x,r)\in\Bd(\bG,E)\times\bbR$, there is
a filtration lattice $\blieG(E)_{x,r}$ and a subgroup $\bG(E)_{x,r}$ if
$\rtm\ge0$ defined by Moy and Prasad \cite{MP94}.
We assume that the valuation is normalized such that
for a tamely ramified Galois extension $E'$ of $E$
and $x\in\Bd(\bG,E)\subset\Bd(\bG,E')$, we have
\[
\blieG(E)_{x,r}=\blieG(E')_{x,r}\cap\blieG(E).
\]
If $r>0$, we also have
\[
\bG(E)_{x,r}=\bG(E')_{x,r}\cap\bG(E).
\]
For simplicity, we put $\lieG_{x,\rtm}:=\blieG(k)_{x,\rtm}$ and $G_{x,\rtm}:=\bG(k)_{x,\rtm}$, etc.
We will also use the following notation. Let $\rtm\in\bbR$ and $x\in\Bd(G)$:
\begin{enumerate}
\item
$\lieG_{x,r^+}:=\cup_{s>r} \lieG_{x,s}$, and if $r\ge 0$,
$G_{x,r^+}:=\cup_{s>r} G_{x,s}$;
\item
$\lieG^\ast_{x,r}:=\left\{\chi\in\lieG^\ast
\mid\chi(\lieG_{x,(-r)^+})\subset\pid_k\right\}$;
\item
$\lieG_r:=\cup_{y\in\Bd(G)} \lieG_{y,r}$ and
$\lieG_{r^+}:=\cup_{s>r} \lieG_s$;
\item
$G_r:=\cup_{y\in\Bd(G)} G_{y,r}$ and
$G_{r^+}:=\cup_{s>r} G_s$ for $r\ge0$.
\item
For any facet $F\subset \Bd(G)$, let $G_F:=G_{x,0}$ for some $x\in F$. Let $[F]$ be the image of $F$ in $\Bd^{\mathrm{red}}(G)$. Then,  let $G_{[F]}$ denote the stabilizer of $[F]$ in $G$. Note that $G_F\subset G_{[F]}$. Similarly, $G_{[x]}$ is the stabilizer of $[x]\in\Bd^{\mathrm{red}}(G)$ in $G$. However, $G_x$ will denote $G_{x,0}$, the parahoric subgroup associated to $x$.
\end{enumerate}


\subsection{\bf Unrefined minimal $\bK$-types and good cosets}
\label{sec: Ktypes}
For simplicity, as in \cite{MP94},
we assume that there is a natural isomorphism
$\iota:G_{x,\rtm}/G_{x,\rtm^+}\longrightarrow\lieG_{x,\rtm}/\lieG_{x,\rtm^+}$
when $\rtm>0$. 
By \cite[(2.4)]{Yu01}, such an isomorphism exists
whenever $\bG$ splits over a tamely ramified extension of $k$
(see also \cite[\S1.6]{Adl98}).

\begin{definition}\label{def: unrefined}
\rm
An {\it unrefined minimal $\bK$-type} (or {\it minimal $\bK$-type})
is a pair $(G_{x,\dpi},\chi)$,
where $x\in\Bd(G)$, $\dpi$ is a nonnegative real number,
$\chi$ is a representation of $G_{x,\dpi}$ trivial on $G_{x,\dpi^+}$ and

(i) if $\dpi=0$, $\chi$ is an irreducible cuspidal representation of $G_{x}/G_{x,0^+}$
inflated to $G_{x}$,

(ii) if $\dpi>0$, then $\chi$ is
a nondegenerate character of $G_{x,\dpi}/G_{x,\dpi^+}$.
\end{definition}

The $\dpi$ in the above definition is called the {\it depth} of
the minimal $\bK$-type $(G_{x,\dpi},\chi)$.
Recall that a coset $X+\lieG^\ast_{x,(-\dpi)^+}$ in $\lieG^\ast$
is nondegenerate if $X+\lieG^\ast_{x,(-\dpi)^+}$ does not contain any
nilpotent element. If a character $\chi$ of $G_{x,\dpi}$ is {\it represented} by $X+\lieG^\ast_{x,(-\dpi)^+}$, i.e. $\chi(g)=\Omega_k(X'(\iota(g))$ with $X'\in X+\lieG^\ast_{x,(-\dpi)^+}$, a character $\chi$ of $G_{x,\dpi}$ is {\it nondegenerate}
if $X+\lieG^\ast_{x,(-\dpi)^+}$ is nondegenerate.

\begin{definition} \rm
Two minimal $\bK$-types $(G_{x,\dpi},\chi)$ and $(G_{x',\dpi'},\chi')$ are
said to be {\it associates} if they have the same depth $\dpi=\dpi'$, and
\begin{enumerate}
\item
if $\dpi=0$, there exists $g\in G$ such that $G_{x}\cap G_{gx'}$
surjects onto both $G_{x}/G_{x,0^+}$ and $G_{gx'}/G_{gx',0^+}$,
and $\chi$ is isomorphic to $\,^{g}\!\chi'$,
\item
if $\dpi>0$, the $G$-orbit of the coset
which realizes $\chi$ intersects the coset which realizes $\chi'$.
\end{enumerate}
\end{definition}

We also recall the definition of good cosets. In \S3, we will
prove some facts concerning good $\bK$-types.
The following is a minor modification of the definition in \cite{AR00} (see also \cite[\S2.4]{Asym1}).

\begin{definition}\label{defn: good}\rm \
\begin{enumerate}
\item
Let $\bT\subset\bG$ be a maximal $k$-torus which splits over a tamely
ramified extension $E$ of $k$. Let $\Phi(\bT,E)$ be the set of
$E$-roots of $\bT$. Then, $X\in\lieT$ is a \emph{good element
of depth $\rtm$} if $X\in\lieT_r\setminus\lieT_{r^+}$
and for any $\alpha\in\Phi(\bT,E)$, $\nu(d\alpha(X))= r$ or $\infty$.
\item
Let $r<0$ and $x\in\Bd(G)$.
A coset $\gsc$ in $\lieG_{x,r}/\lieG_{x,r^+}$
is {\it good} if there is a good
element $X\in\lieG$ of depth $r$ such that
$\gsc=X+\lieG_{x,r^+}$ and $x\in\Bd(\bC_{\bG}(X),k)$.
\item
A minimal $\bK$-type $(G_{x,\dpi},\chi)$ with $\dpi>0$ is {\it good} if the associated dual coset is good.
\end{enumerate}
\end{definition}

\subsection{Generic $G$-datum}\label{sub:generic-datum}

Yu's construction of supercuspidal representations
starts with a {\it generic $G$-datum}, which consists of five components.
Recall $\bG'\subset \bG$ is
a \emph{tamely ramified twisted Levi subgroup}
if $\bG'(E)$ is a Levi subgroup of $\bG(E)$ for a tamely ramified
extension $E$ of $k$.

\begin{definition}\label{defn: generic G-datum}\rm
A {\it generic $G$-datum} is a quintuple
$\datum=(\vec\bG,\xo,\vec\rtm,\vec\phi,\rho)$ satisfying the following:

\item{${\mathbf{D}}1.$}
$\vec{\bG}=(\bG^0\subsetneq\bG^1\subsetneq\cdots\subsetneq\bG^d=\bG)$ is a tamely ramified twisted
Levi sequence such that $\bZ_{\bG^0}/\bZ_{\bG}$ is anisotropic.

\midvsp

\item{${\mathbf{D}}2.$}
$\xo\in\Bd(\bG^0,k)$.

\midvsp

\item{${\mathbf{D}}3.$}
$\vec\rtm=(\rtm_0,\rtm_1,\cdots,\rtm_{d-1},\rtm_d)$ is
a sequence of positive real numbers
with $0<\rtm_0<\cdots<\rtm_{d-2}< \rtm_{d-1}\le\rtm_d$ if $d>0$, and
$0\le\rtm_0$ if $d=0$.

\midvsp

\item{${\mathbf{D}}4.$}
$\vec\phi=(\phi_0,\cdots,\phi_d)$ is a sequence of quasi-characters,
where $\phi_i$ is a generic quasi-character of $G^i$ (see \cite[\S9]{Yu01} for the definition
of generic quasi-characters);
$\phi_i$ is trivial on $G^i_{\xo,\rtm_i^+}$, but
non-trivial on $G^i_{\xo,\rtm_i}$ for $0\le i\le d-1$.
If $\rtm_{d-1}<\rtm_d$,
then $\phi_d$ is trivial on
$G^d_{\xo,\rtm{}^+_d}$ and nontrivial on $G^d_{\xo,\rtm_d}$, and otherwise if $r_{d-1}=r_d$, then $\phi_d=1$.

\midvsp

\item{${\mathbf{D}}5.$}
$\rho$ is an irreducible representation of $G^0_{[\xo]}$,
the stabilizer in $G^0$ of the image $[\xo]$ of $\xo$
in the reduced building of $\bG^0$,
such that $\rho|G^0_{\xo,0^+}$ is isotrivial
and $c\textrm{-Ind}_{G^0_{[\xo]}}^{G^0}\rho$ is irreducible and supercuspidal.
\end{definition}

\begin{rem}\label{rem: G-datum} \

\begin{enumerate}
\item
By (6.6) and (6.8) of \cite{MP96},
${\mathbf{D}}5$ is equivalent to the condition that
$G^0_{\xo}$ is a maximal parahoric subgroup in $G^0$ and
$\rho|G^0_{\xo}$ induces a cuspidal representation of
$G^0_{\xo}/G^0_{\xo,0^+}$.
\item
Recall from \cite{Yu01} that there is a canonical sequence of embeddings
$\Bd(\bG^0,k)\hookrightarrow\Bd(\bG^1,k)\hookrightarrow
\cdots\hookrightarrow\Bd(\bG^d,k)$.
Hence, $\xo$ can be regarded as a point of each $\Bd(\bG^i,k)$.
\item
There is a finite number of pairs $(\vec\bG, x)$ up to $G$-conjugacy, which arise in a generic $G$-datum: By \S1.2 in \cite{KY11}, there are finitely many choices for $\vec \bG$ up to $G$-conjugacy. In particular,  there are finitely many choices for $\bG^0$, and for each $\bG^0$ the number of vertices in $\Bd(G^0)$ is finite up to $G^0$-conjugacy.
\end{enumerate}

\end{rem}

\subsection{Construction of $J_{\datum}$}
\label{sub:notation yu}
Let $\datum=(\vec\bG,\xo,\vec\rtm,\vec\phi,\rho)$ be a generic $G$-datum. Set $s_i:=r_i/2$ for each~$i$. Associated to $\vec\bG,\ \xo$ and $\vec\rtm$, we define the following open compact subgroups.
\begin{enumerate}
\item
$\rK^0:=G^0_{[\xo]}$ ; $\rK^{0}_+:=G^0_{\xo,0^+}$.
\item
$\rK^i:= G^0_{[\xo]}G^1_{\xo,s_0}\cdots G^i_{\xo,s_{i-1}}$ ;
$\rK^{i}_+:= G^0_{\xo,0^+}G^1_{\xo,s_0^+}\cdots G^i_{\xo,s_{i-1}^+}$ \
for $1\le i\le d$.

\item
$J^i:=(\bG^{i-1},\bG^i)(k)_{\xo,(\rtm_{i-1},\stm_{i-1})}$ ;
$J^i_+:=(\bG^{i-1},\bG^i)(k)_{\xo,(\rtm_{i-1},\stm_{i-1}^+)}$ in the notation of \cite[\S1]{Yu01}.
\end{enumerate}
For $i>0$, $J^i$ is a normal subgroup of $\rK^i$ and we have $\rK^{i-1} J^i=\rK^i$ (semi-direct product). Similarly $J^i_+$ is a normal subgroup of $\rK^i_+$ and $\rK^{i-1}_+ J^i_+=\rK^{i}_+$.  Finally let $J_{\datum}:=K^d$ and $J_+:=K^{d}_+$, and also $s_\Sigma:=s_{d-1}$ and $r_\Sigma:=r_{d-1}$. When there is no confusion, we will drop the subscript $\Sigma$ and simply write $J,\ r,\ s$, etc.

\subsection{Construction of $\rho_{\datum}$} \label{sub:yusc}

One can define the character $\hat\phi_i$ of $\rK^0G^i_{\xo}G_{\xo,\stm_i^+}$
extending $\phi_i$ of $\rK^0G^i_{\xo}\subset G^i$.
For $0\le i<d$, there exists by the Stone-von Neumann theorem a representation $\tilde\phi_i$ of $\rK^i\ltimes J^{i+1}$ such that $\tilde\phi_i|J^{i+1}$ is $\hat\phi_i|J^{i+1}_+$-isotypical and $\tilde\phi_i|\rK^{i}_+$ is isotrivial.

Let $\textrm{inf}(\phi_i)$ denote the inflation of $\phi_i|\rK^i$
to $\rK^i\ltimes J^{i+1}$. Then $\mathrm{inf}(\phi_i)\otimes\tilde\phi_i$
factors through a map
\[
\rK^i\ltimes J^{i+1}\longrightarrow \rK^i\rJ^{i+1}=\rK^{i+1}.
\]
Let $\kappa_i$ denote the corresponding representation of $\rK^{i+1}$.
Then it can be extended trivially to $\rK^d$,
and we denote the extended representation again by $\kappa_i$
(in fact $\kappa_i$ could be further extended to the semi-direct product $K^{i+1}G_{x,s_{i+1}}\supset K^d$ by making it trivial on $G_{x,s_{i+1}}$). Similary we extend $\rho$ from $G^0_{[x]}$ to a representation of $\rK^d$ and denote
this extended representation again by $\rho$.
Define a representation $\kappa$ and $\rho_{\datum}$
of $\rK^d$ as follows:
\begin{equation}\label{eq: construction-sc}
\begin{array}{ll}
\kappa&:=\kappa_0\otimes\cdots\otimes\kappa_{d-1}\otimes(\phi_d|\rK^d),\\
\rho_{\datum}&:=\rho\otimes\kappa.
\end{array}
\end{equation}
Note that $\kappa$ is defined only from $(\vec\bG,\xo,\vec\rtm,\vec\phi)$
independently of $\rho$.

\begin{rem}
One may construct $\kappa_i$ as follows:
set $J_1^i:=G^i_{x,0^+}G_{x,s_i}$ and $J_2^i:=G^i_{x,0^+}G_{x,s_i^+}$.
Write also $\hat\phi_i$ for the restriction of $\hat\phi_i$ to $J^i_2$.
Then,  one can extend $\hat\phi_i$ to
$J^i_1$ via Heisenberg representation and to $G^i_{[y]}G_{x,s_i}$
by Weil representation upon fixing a special isomorphism (see \cite{Yu01} for details):
\[
\begin{array}{lclcl}
J_2^i&\rightarrow&J_1^i&\rightarrow&G^i_{[y]}G_{x,s_i}\\
\hat\phi_i& &\rho_{\hat\phi_i}& &\omega_{\hat\phi_i}.
\end{array}
\]
Note that we have inclusions $J_\Sigma=K^d\subset G^i_{[\xo]}G_{\xo,s_i^+}\subset G_{[\xo]}$, and we have
$\kappa_i\simeq\omega_{\hat\phi_i}|K^d$.
\end{rem}

\begin{thm} [Yu]\label{t:Yu}
$\pi_{\datum}=c\textrm{-}\mathrm{Ind}_{J_\Sigma}^G\rho_{\datum}$
is irreducible and thus supercuspidal.
\end{thm}

\begin{rem}\label{rmk: minimal type} Let $\datum$ be a generic $G$-datum. If $G$ is semisimple, comparing Moy-Prasad minimal $K$-types and Yu's constructions, we observe the following:
\begin{enumerate}
\item The depth of $\pi_{\datum}$ is given $r_\datum=r_d=r_{d-1}$. (Even if $G$ is not semisimple, the depth is $r_d$, cf. \cite[Remark 3.6]{Yu01}, but it may not equal $r_{d-1}$.)
\item $(G_{\xo,r_{d-1}},\phi_{d-1})$ is a good minimal $K$-type of $\pi_{\datum}$ in the sense of \cite{Asym1}.
\end{enumerate}
\end{rem}

\subsection{Supercuspidal representations via compact induction}\label{sub:sc-cpt-ind}

Denote by $\Irr(G)$ the set of (isomorphism classes of) irreducible smooth representations of $G$. Fix a Haar measure on $G$. Write $\Irr^2(G)$ (resp. $\Irr^\scusp(G)$) for the subset of square-integrable (resp. supercuspidal) members. For each $\pi\in \Irr^2(G)$ let $\fdeg(\pi)$ denote the formal degree of $\pi$.  For each $\pi\in \Irr(G)$, $\Theta_\pi$ is the Harish-Chandra character, which is in $L^1_{\mathrm{loc}}(G)$ and locally constant on $G_\reg$.

  Define $\Irr^{\Yu}(G)$ to be the subset of $\Irr^\scusp(G)$ consisting of all supercuspidal representations which are constructed by Yu, namely of the form $\pi_\Sigma$  as above. Write $\Irr^{\cind}(G)$ for the set of $\pi\in\Irr^\scusp(G)$ which are compactly induced, meaning that there exist an open compact-mod-center subgroup $J\subset G$ and an irreducible admissible representation $\rho$ of $J$ such that $\pi\simeq \cind^{G}_J(\rho)$. We have that
  \[\Irr^{\Yu}(G) \subset \Irr^{\cind}(G) \subset \Irr^\scusp(G).\]
  The first inclusion is a consequence of Yu's theorem (Theorem \ref{t:Yu}) and generally strict. A folklore conjecture asserts that the second inclusion is always an equality. It has been verified through the theory of types for $\GL_n$ and $\SL_n$ by Bushnell-Kutzko, for inner forms of $\GL_n$ by Broussous and S\'echerre-Stevens, and for $p$-adic classical groups by S. Stevens when $p\neq2$ (\cite{BK93}, \cite{BK94}, \cite{Bro98}, \cite{SS08} and \cite{Ste08}). In general,
  according to the main result of \cite{Kim07}, there exists a lower bound $p_0=p_0(k,G)$ (depending on $k$ and $G$) such that both inclusions are equalities if $p\ge p_0$. Precisely this is true for every prime $p$ such that the hypotheses $\Hypk$, (HB), (HGT), and $\HypN$ of \cite[\S3.4]{Kim07} are satisfied.

\subsection{Hypotheses}\label{sub:hypo}
The above hypotheses will be assumed in a variety of our results in the next two sections. We will clearly state when the hypotheses are needed. As they are too lengthy to copy here, the reader interested in the details is referred to \cite[\S3.4]{Kim07}. For our purpose it suffices to recall the nature of those hypotheses: $\Hypk$ is about the existence of filtration preserving exponential map, (HB) is to identify $\lieG$ and its linear dual $\lieG^\ast$, (HGT) is about the abundance of good elements and $\HypN$ is regarding nilpotent orbits.

\subsection{Formal degree} Recall that $\deg(\pi)$ denotes the formal degree of $\pi$.
\begin{lem}\label{l:bound-on-d} Let $\Sigma$ be a generic $G$-datum. Then

\benu
\item[(i)] $\deg(\pi_\Sigma)=\dim(\rho_\Sigma)/\vol_{G/Z_G}(J_\Sigma/Z_G)$.
\item[(ii)]  $\frac{1}{\vol_{G/Z_G}(J_\Sigma/Z_G)}\le \deg(\pi_\Sigma)\le \frac{q^{\dim(\lieG)}}{\vol_{G/Z_G}(J_\Sigma/Z_G)}.$
\eenu
\end{lem}

\begin{proof}
Assertion (i) is easily deduced from the defining equality for $\deg(\pi_\Sigma)$:
\[
\deg(\pi_\Sigma)\int_{G/Z_G}\Theta_{\rho_\Sigma}(g)\ol{\Theta_{\rho_\Sigma}(g)} dg = \dim \rho_\Sigma,
\]
cf. \cite[Thm A.14]{BH96}.

(ii) Let $\rho$ and $\kappa$ be as in \eqref{eq: construction-sc}. One sees from the construction of supercuspidal representations that
$\dim(\rho)\le[G^0_{\xo}:G^0_{\xo,0^+}]$, and the dimension formula for finite Heisenberg representations yields
\[
\begin{aligned}
\dim(\kappa_i)=[J^{i+1}: J^{i+1}_+]^{\frac12}&=[(\lieG^{i},\lieG^{i+1})_{x,(\rtm_{i},\stm_{i})}: (\lieG^{i},\lieG^{i+1})_{x,(\rtm_{i},\stm_{i}^+)}]^{\frac12}\le[\mathbf{\mathfrak g}^{i+1}(\mathbb F_q):\mathbf{\mathfrak g}^{i}(\mathbb F_q))]^{\frac12}, \\
\dim(\kappa)&=\prod_{i=i}^{d-1}\dim(\kappa_i)\le [\mathbf{\mathfrak g}(\mathbb F_q):\mathbf{\mathfrak g^0}(\mathbb F_q))]^{\frac12}
\end{aligned}
\]
Hence,
\[1\le\dim(\rho_\Sigma)=\dim(\rho)\dim(\kappa)\le
[G^0_{\xo}:G^0_{\xo,0^+}][\mathbf{\mathfrak g}(\mathbb F_q):\mathbf{\mathfrak g^0}(\mathbb F_q))]^{\frac12}\le q^{\dim(\lieG)}\]

\end{proof}

There is no exact formula yet known for the formal degree $\deg(\pi_\Sigma)$ of tame supercuspidals, or equivalently for $\vol_{G/Z_G}(J_\Sigma/Z_G)$ and $\dim(\rho_\Sigma)$, which is also an indication of the difficulty in computing the trace character $\Theta_{\pi_\Sigma}(\gamma)$ in this generality since $\deg(\pi_\Sigma)$ appears as the first term in the local character expansion. In this direction a well-known conjecture of Hiraga-Ichino-Ikeda~\cite{Hiraga-Ichino-Ikeda} expresses $\deg(\pi_\Sigma)$ in terms of the Langlands parameter conjecturally attached to $\pi_\Sigma$.

For our purpose it is sufficient to know that
$\frac1{r_\Sigma} \log_q  \vol_{G/Z_G}(J_\Sigma/Z_G)$ is bounded above and below as $\Sigma$ varies by constants depending only on $G$, which follows from Lemma~\ref{l:bound-on-d}.
Below we shall use similarly that $\frac 1{r_{\Sigma}}\log_q \vol_{G/ Z_G}(L_{s_\Sigma})$ is bounded above and below, where $L_{s_\Sigma}:=G^{d-1}_{[x]}G_{s_\Sigma}$. Note that $J_\Sigma\subset L_{s_\Sigma}$.

\section{Preliminary lemmas on Moy-Prasad subgroups}
In this subsection, we prove technical lemmas that we need to prove the main theorem.
We keep the notation from the previous section.

\subsection{Lemmas on $\pi_\Sigma$ and $\gamma$}

Recall from \cite{Asym1} that when (HB) and (HGT) are valid, any irreducible smooth representation $(\tau,V_\tau)$ contains a good minimal $K$-type. The following lemma analyzes other possible minimal $K$-types occurring in $(\tau,V_\tau)$.

\begin{lem}\label{lem: good type} Suppose (HB) and (HGT) are valid.
Let $(V_\tau,\tau)$ be an irreducible smooth representation of $G$ of positive depth $\dpi$. Let $(\chi,G_{x,\dpi})$ be a good minimal $K$-type of $\tau$ represented by $\Xs+\lieG_{x,(-\dpi)^+}$ where $\Xs\in \lieG_{x,(-\dpi)}$ is a good element of depth $(-\dpi)$. Let $G'$ be the connected component of the centralizer of $\Xs$ in $G$. \begin{enumerate}
\item[(1)] Fix an embedding $\Bd(G')\hookrightarrow\Bd(G)$ (such an embedding can be chosen by \cite[Thm 2.2.1]{Lan00}) and let $C'_{x}$ be a facet of maximal dimension in $\Bd(G')$ containing $x$ in its closure $\overline C'_x$.
 There exists a facet of maximal dimension $C_{x}$ in $\Bd(G)$ such that $x\in\overline C_{x}\cap \overline C'_{x}$ and $C'_{x}\cap \overline C_{x}$ is of maximal dimension in $\Bd(G\rq)$.
\item[(2)] Let $y\in \Bd(G)$ and suppose that $V_{\tau}^{G_{y,\dpi^+}}\neq 0$.
As a representation of $G_{y,\dpi}$,  $V_\tau^{G_{y,\dpi^+}}$  is a sum of characters $\chi'$'s which are represented by $^h(\Xs+\eta')+\lieG_{y,(-\dpi)^+}\subset \lieG_{y,(-\dpi)}$ for some $\eta'\in \lieG'_{(-\dpi)^+}$ and $h\in G_{[y]}S_{\xo}$ for some compact mod center set $S_{\xo}$. Moreover, one can choose $S_{\xo}$ in a way depending only on $x, G\rq{}$.
\end{enumerate}
\end{lem}

\begin{remark}
Note that $x\in\Bd(G')$ by \cite[Theorem 2.3.1]{Asym1}.
\end{remark}

\begin{proof}
(1) Let $V:=\cup_C \overline C$ where the union runs over the set of facets of maximal dimension $C\subset\Bd(G)$ with $x\in\overline C$. Let $V^\circ$ be the interior of $V$. Then, $x\in V^\circ$ and $V^\circ$ is open in $\Bd(G)$, hence, $C\rq_{x}\cap V^\circ\neq\emptyset$ and is open in $\Bd(G\rq)$. Therefore, at least one of $\overline C\cap C_x\rq$ with $x\in\overline C$ contains an open set in $\Bd(G\rq)$. Set $C_x$ to be one of such facets.

(2) Since the action of $G_{y,\dpi}$ on $V_\tau^{G_{y,\dpi^+}}$ factors through the finite abelian quotient $G_{y,\dpi}/G_{y,\dpi^+}$, we see that $V_\tau^{G_{y,\dpi^+}}$ decomposes as a direct sum of characters of $G_{y,\dpi}$. Let $\chi'$ be a $G_{y,\dpi}$ subrepresentation of $V_\tau^{G_{y,\dpi^+}}$. Then, $(\chi',G_{y,\dpi})$ is also a minimal $K$-type of $\tau$.
Let $\Xs'+\lieG_{y,(-\dpi)^+}\subset\lieG_{y,(-\dpi)}$ be the dual cosets representing $\chi'$. Then, $(\Xs+\lieG_{x,(-\dpi)^+})\cap \,^G\!(\Xs'+\lieG_{y,(-\dpi)^+})\ne\emptyset$. Since $(\Xs+\lieG_{x,(-\dpi)^+})=\,^{G_{x,0^+}}(\Xs+\lieG'_{x,(-\dpi)^+})$, there are $\eta\in\lieG'_{x,(-\dpi)^+}$ and $g\in G$ such that $\Xs+\eta\in \,^{g^{-1}}(\Xs'+\lieG_{y,(-\dpi)^+})\subset \lieG_{g^{-1}y,-\dpi}$. By \cite[Lem 2.3.3]{Asym1}, $g^{-1}y\in\Bd(G')$.

To choose $S_{\xo}$, let $\Apt(T)$ be  an apartment in $\Bd(G)$ such that $C_{x}\cup C'_{x}\subset \Apt(T)$. For each alcove $C\subset\Apt(T)$ with $\overline C\cap \overline C'_{x}\neq\emptyset$, choose $w_C\in N_G(T)$ such that $ C=w_C C_{x}$. Now, set
\[S_{\xo}:=\{\delta\cdot w_C^{-1}\mid C\textrm{ is an alcove in }\Bd(G)\textrm{ with }\overline C\cap \overline C'_{x}\neq\emptyset, \ \delta\in G_{[C_{x}]}\}.\]
We claim that there is $g'\in G'$ such that $gg'\in G_{[y]} S_{\xo}$. Let $g'\in G'$ such that $(gg')^{-1}y\in\overline C'_{x}$. Then, there is $g''{}^{-1}\in S_{\xo}$ such that $g''y=(gg')^{-1}y$. Hence, $gg'g''\in G_{[y]}$ and $gg'\in G_{[y]}S_{\xo}$. Then one can take $h=gg'$ and $\eta'=g'^{-1}\eta g'$ since ${}^h(\Xs+\eta')\equiv{}^g ({}^{g'}\Xs+\eta)\equiv{}^g (\Xs+\eta) \equiv\Xs' \pmod{\fkg_{y,(-\dpi)^+}}$.
By construction $S_{\xo}$ depends only on $x$ and $G\rq{}$.
\end{proof}

\begin{defn}\label{d:singular-depth}
Let $\gamma\in G_\reg$. Let $\bT^\gamma$ be the unique maximal torus containing $\gamma$, and $\Phi:=\Phi(\bT^\gamma)$ the set of absolute roots of $\bT^\gamma$. Let $\Phi^+:=\Phi^+(\bT^\gamma)$ be the set of positive roots.
\begin{enumerate}
\item[(i)]
Define the \emph{singular depth} $\sd_\alpha(\gamma)$ of $\gamma$ in the direction of $\alpha\in\Phi$ as
\[
\sd_\alpha(\gamma):=\nu(\alpha(\gamma)-1).
\]
and the \emph{singular depth} $\sd(\gamma)$ of $\gamma$ as
\[
\sd(\gamma):=\max_{\alpha\in\Phi}\sd_\alpha(\gamma).
\]
When $\gamma$ is not regular, see \cite[\S4]{AK07} for definition. In \cite[\S4.2]{Meyer-Solleveld:growth}, $\sd(\gamma)$ is defined as $\max_{\alpha\in\Phi^+}\sd_\alpha(\gamma)$. When $\gamma$ is compact, both definitions coincide.

\smallskip

\item[(ii)]
Recall that the height of $\alpha\in\Phi^+(\bT^\gamma)$ is defined inductively as follows:

\smallskip

 $\bullet$ $\hg(\alpha)=1$ if $\alpha\in\Phi^+$ is simple; \\
\indent $\bullet$ $\hg(\alpha+\beta)=\hg(\alpha)+\hg(\beta)$ if $\alpha,\beta,\alpha+\beta\in\Phi^+$.

 \smallskip

\noindent
Define the \emph{height $h_G$ of $\Phi$} as $\max_{\alpha\in\Phi^+}\hg(\alpha)$. Note that the height of $\Phi$ depends only on $G$.

\end{enumerate}

\end{defn}

\begin{lem}\label{lem: centralizer1}
Suppose $\gamma\in G_{\rs}\cap T^\gamma_0$ splits over a tamely ramified extension.
Suppose $z\in\Apt(T^\gamma)$, and $^g\!\gamma\in G_{z}$ for $g\in G$.
Then, $^g\!T^\gamma_{h_\gamma^+}\subset G_{z}$, where $h_\gamma:=h_G\cdot \sd(\gamma)$.
\end{lem}


\begin{proof}
This is a reformulation of \cite[Lemma 4.3]{Meyer-Solleveld:growth}. More precisely, $^g\!\gamma\in G_{z}$ is equivalent to $z\in \Bd(G)^{^g\!\gamma}$, hence $z\in \Bd(G)^{^g\!(\gamma T_{h_\gamma^+})}$ by {\it loc. cit.} , which in turn implies that $^g\!(\gamma T_{h_\gamma^+})\subset G_{z}$ and $^u\!T_{h_\gamma^+}\subset G_{z}$.
\end{proof}

\begin{lem}\label{lem: centralizer2}
Suppose $\gamma\in G_{\rs}\cap G_0$ splits over a tamely ramified extension $E$.
Let $\bT^\gamma \subset \bG^0\subsetneq \bG^1\subsetneq\cdots\subsetneq \bG^d=\bG$ be an $E$-twisted tamely ramified Levi sequence, $z\in\Apt(\bT^\gamma,E)\subset \Bd(\bG^0,E)$ and $a_i\in \bbR$ with $0\le a_1\le \cdots\le a_d$. Set $K_E=\bG^0(E)_{z}\bG^1(E)_{z,a_1}\bG^2(E)_{z,a_2}\cdots \bG^d(E)_{z,a_d}$. Suppose $g\in G$ such that $^g\gamma\in K_E$. Then, we have $^g\!\bT^\gamma(E)_{\aconst^+}\subset K_E$ where $\aconst=h_\gamma+a_d$ and $h_\gamma$ is as in Lemma \ref{lem: centralizer1}.
\end{lem}

\proof
Without loss of generality, we may assume $E=k$ and we will for simplicity of notation.
 Let $O\in\Apt(T^\gamma)$ defined by $\alpha(O)=0$ for all $\alpha\in \Phi$. For $\alpha\in\Phi$, let $U_\alpha$ be the root subgroup associated to $\alpha$. We fix the pinning $x_\alpha: k\rightarrow U_\alpha$. Define $U_{\alpha,r}$ to be the image $\{u\in k\mid \val(u)\ge r\}$ under the isomorphism $x_\alpha: k\rightarrow U_\alpha$.  Let $C\subset \Apt(T^\gamma)$ be the facet of maximal dimension  with $O\in\overline C$. One may assume that $z\in \overline C$ (by conjugation by an element of $N_G(T^\gamma)$ if necessary).
Then, $G_C$ is an Iwahori subgroup and we have the Bruhat decomposition $G=G_CN_{G}(T^\gamma)G_C$, cf. \cite[3.3.1]{Tit79}. Let $w\in N_{G}(T^\gamma)$ with $g\in G_C w G_C$.
Let $A_w=\{\alpha\in\Phi(T^\gamma)\mid U_{w\alpha}\cap G_C\subsetneq w(U_\alpha\cap G_C)\}$. Write $U_w=\prod_{\alpha\in A_w} U_{\alpha}$. Note that $U_w$ may not necessarily be a group.
We prove the lemma through steps (1)-(8) below.

\medskip

(1) Let $w\in N_G(T^\gamma)$. Then, there is a Borel subgroup $B$ containing $wU_ww^{-1}$.

\begin{proof}
Each chamber $D$ containing $O$ in its closure defines an open cone $\mathcal C_D$ in $\Apt(T^\gamma)$ and we have $\Apt(T^\gamma)=\cup _{O\in \overline D}\overline{\mathcal C}_D$ is the union is over the chambers $D$ with $O\in\overline D$. Recall that each $\mathcal C_D$ defines a Borel subgroup $B_D$. If $\mathcal C_D$ contains $wC$, one can take $B=B_D$.
\end{proof}

\medskip

For $D$ as in (1), let $\Phi^+_{D}$ be the set of positive $T^\gamma$-roots associated to $\mathcal C_D$ and write $\hg_D=\hg_{\Phi^+_D}$ for simplicity.

\medskip

(2) Write $B_D=T^\gamma U$ and  let $\overline U$ be the opposite unipotent subgroup. Then, $gw^{-1}=t\cdot \overline u\cdot u$ for $t\in T^\gamma_0$, $\overline u\in G_C\cap\overline U$ and $u\in U$.

\begin{proof}
We have that $G_C$ has an Iwahori decomposition with respect to $B$: $G_C=(G_C\cap T^\gamma)(G_C\cap \overline U)(G_C\cap U)$.  Then, the above follows from $gw^{-1}\in G_C w G_C w^{-1}\subset G_C w U_w w^{-1}\subset (G_C\cap T^\gamma)(G_C\cap \overline U)(G_C\cap U) U\subset (G_C\cap T^\gamma)(G_C\cap \overline U)U$.
\end{proof}

(3) Since $t\in T^\gamma_0\subset K_E$, we may assume without loss of generality that $t=1$ or $gw^{-1}=\overline u u$ .

\medskip

(4) We have $u\in G_{x+\sd(\gamma)\hg_{D}}\cap U$.

\begin{proof}
Observe that $^w\!\gamma\in K_E$ and $\sd({}^w\!\gamma)=\sd(\gamma)$.
Observe also that $^u\!({}^w\gamma)\in G_C\subset G_z$ since $\overline u\in G_z$ and $K_E\subset G_z$.
Since $^u\!({}^w\gamma)\in G_C\subset G_z$ we can apply~\cite[Proposition 4.2]{Meyer-Solleveld:growth} to deduce that $u\in G_{z+\sd(\gamma)\hg_{D}}\cap U$.
\end{proof}

\medskip

(5) For $\gamma'\in T^\gamma_{A^+}$, we have (\emph{i}) $(u,{}^w\!(\gamma\gamma'))\equiv (u,{}^w\!\gamma)\pmod{G_{z+\sd(\gamma)\hg_D,A^+}\cap U}$, and
(\emph{ii})  $(u,{}^w\!(\gamma\gamma'))\equiv (u,{}^w\!\gamma)\pmod{G_{z,s}\cap U}$.

\begin{proof}
For $\gamma'\in T^\gamma_{A^+}\subset G_{z+\sd(\gamma)\hg_D,A^+}$, the commutators $(u,{}^w\!(\gamma\gamma'))$ and $(u,{}^w\!\gamma)$ are in $G_{z+\sd(\gamma)\hg_{D}}\cap U$ and also in the same coset mod ${G_{z+\sd(\gamma)\hg_D,A^+}\cap U}$. Hence, ({\it i}) follows.

The assertion ({\it ii}) follows from ({\it i}). Indeed, we note that
\[
G_{z+\sd(\gamma)\hg_D,A^+}\cap U=\prod\limits_{\alpha\in \Phi^+_D} U_{\alpha,-\alpha(z)-\sd(\gamma)\hg_D(\alpha)+A^+}
\]
is contained in $\prod\limits_{\alpha\in \Phi^+_D} U_{\alpha,-\alpha(z)+s^+}$ which is itself contained in $G_{z,s}\cap U$.
\end{proof}


\medskip

(6) For any $\gamma'\in T^\gamma_{s^+}$, we have $^w\!\gamma\overline u^{-1}{}^w\gamma^{-1}\equiv ({}^w\!(\gamma\gamma'))\overline u^{-1}({}^w\!(\gamma\gamma'))^{-1}\pmod{G_{z,s}}$.

\medskip

(7) For any $\gamma'\in T^\gamma_{A^+}$, we have $g\gamma g^{-1}{}^w\gamma^{-1}\equiv g(\gamma\gamma') g^{-1}{{}^w\!(\gamma\gamma')}^{-1}\pmod{G_{z,s}}$

\begin{proof}
Write
$g\gamma g^{-1}{}^w\gamma^{-1}=\overline u (u,{}^w\gamma){}^w\gamma\overline u^{-1}{}^w\gamma^{-1}$ and
\[
g(\gamma\gamma') g^{-1}{}^w\!(\gamma\gamma')^{-1}=\overline u (u,{}^w\!(\gamma\gamma')){}^w\!(\gamma\gamma')\overline u^{-1}{}^w\!(\gamma\gamma')^{-1}.
\]
Then, (7) follows from (5) and (6).
\end{proof}

(8) $^gT_{A^+}^\gamma\subset K_E$.

\begin{proof}
Note that ${}^w\!\gamma, {}^w\!(\gamma\gamma')\in K_E$. Hence, $g\gamma g^{-1}\equiv g(\gamma\gamma') g^{-1}\pmod{K_E}$ by (6). Since $g\gamma g^{-1}\in K_E$, we have $g T^\gamma_{A^+} g^{-1}\subset {K_E}$.
\end{proof}

The proof of Lemma \ref{lem: centralizer2} is now complete.
\qed

\begin{lem}\label{lem: centralizer2-rank1}
In the same situation as in Lemma \ref{lem: centralizer2}, suppose in addition that $h_G=1$. Then, we have $^g\!\bT^\gamma(E)_{\sd(\gamma)^+}\subset K_E$.
\end{lem}

\proof Since $h_G=1$, we have $d=0$ or $d=1$. If $d=0$, the assertion is precisely Lemma \ref{lem: centralizer1}. We assume $d=1$ from now. Then, $K_E=T^\gamma_0G_{z,a_1}$.

As before, we may assume $E=k$. The assertions (1)-(8) below refer to those in the above proof of Lemma~\ref{lem: centralizer2}. Following the proof of Lemma \ref{lem: centralizer2}, write $\Phi_D=\{\pm\alpha\}$. Let $gw^{-1}=\overline u u$ for $u\in U_\alpha$ and $\overline u\in U_{-\alpha}$ as in (3). Under the isomorphism via $x_\alpha:k\rightarrow U_\alpha$, we will use the same notation for $u$ and $x_\alpha^{-1}(u)$. Then, we can write $x_\alpha(u)=u$. Similarly, $x_{-\alpha}(\overline u)=\overline u$. Let $\alpha^\vee :k^\times \rightarrow T^\gamma$ be the coroot of $\alpha$. It is enough to prove that $g(\gamma\gamma') g^{-1}\,^w\!(\gamma\gamma')^{-1}\in K_E$ for any $\gamma'\in T^\gamma_{\sd(\gamma)^+}$.

We have $(u,\,^w\!(\gamma\gamma'))=x_\alpha((1-\alpha(\,^w\!(\gamma\gamma')))u)$ and
$(\overline u,\,^w\!(\gamma\gamma'))=x_{-\alpha}((1-\alpha(\,^w\!(\gamma\gamma')^{-1}))\overline u)$. For simplicity, write $u_{\gamma'}=(1-\alpha(\,^w\!(\gamma\gamma')))u=x_\alpha(u_{\gamma'})$
and $\overline u_{\gamma'}=(1-\alpha(\,^w\!(\gamma\gamma')^{-1}))\overline u=x_{-\alpha}(\overline u_{\gamma'})$.
Since $x_\alpha(u_{\gamma'})\in U_\alpha\cap G_z$ by (4) and Lemma \ref{lem: centralizer1} and $x_{-\alpha}(\overline u)\in U_{-\alpha}\cap G_z$ , we have
\[
\val(u_{\gamma'})\ge -\alpha(z), \quad \val(\overline u)\ge\alpha(z).
\]
Similarly as in (7), we calculate $g(\gamma\gamma') g^{-1}\,^w\!(\gamma\gamma')^{-1}$, but, now explicitly using the Chevalley basis. Then,
\begin{align*}
&g(\gamma\gamma') g^{-1}{}^w\!(\gamma\gamma')^{-1}=\overline u (u,{}^w\!(\gamma\gamma')){}^w\!(\gamma\gamma')\overline u^{-1}{}^w\!(\gamma\gamma')^{-1}\\
&=x_{-\alpha}(\overline u)x_\alpha(u_{\gamma'})x_{-\alpha}(-\overline u)x_{-\alpha}(\overline u) \,\Ad({}^w(\gamma\gamma'))( x_{-\alpha}(-\overline u)) \\
&= x_\alpha(u_{\gamma'}(1+\overline u u_{\gamma'})^{-1})\alpha^\vee ((1+\overline u u_{\gamma'})^{-1})x_{-\alpha}(-\overline u^2 u_{\gamma'}(1+\overline u u_{\gamma'})^{-1})x_{-\alpha}(\overline u_{\gamma'})\\
\end{align*}
When $\gamma'=1$, we have
\[
g\gamma g^{-1}{}^w\!\gamma^{-1}= x_\alpha(u_{1}(1+\overline u u_{1})^{-1})\alpha^\vee ((1+\overline u u_{1})^{-1})x_{-\alpha}(-\overline u^2 u_{1}(1+\overline u u_{1})^{-1})x_{-\alpha}(\overline u_{1}).
\]
Since $g\gamma g^{-1}{}^w\!\gamma^{-1}\in K_E$, we have $\val(1+\overline u u_{-1})=0$ and $\val(u_{1})=\val(u_{1}(1+\overline u u_{1})^{-1})\ge -\alpha(z)+a_1$. Combining this with $\val(\overline u)\ge\alpha(z)$, we have $\val(\overline u^2 u_{1}(1+\overline u u_{1})^{-1})\ge\alpha(z)+a_1$.
Note that

\smallskip

({\it i}) $\val(1+\overline u u_{\gamma'})=\val(1+\overline u u_{1})$ and $\val(u_1)=\val(u_{\gamma'})$;

({\it ii}) $\val(\overline u^2 u_{1}(1+\overline u u_{1})^{-1})=\val(\overline u^2 u_{\gamma'}(1+\overline u u_{\gamma'})^{-1})\ge\alpha(z)+a_1$;

({\it iii}) $\val(\overline u_1)=\val(\overline u_{\gamma'})\ge \alpha(z)+a_1.$ The last inequality follows from $x_{-\alpha}(\overline u_{1})\in K_E$.

\smallskip

\noindent
From ({\it i})-({\it iii}), we have $g(\gamma\gamma') g^{-1}{}^w\!(\gamma\gamma')^{-1}\in K_E$ and conclude that $^g\!T^\gamma_{\sd(\gamma)^+}\subset K_E$.
\qed

\begin{prop}\label{prop: centralizer0} Recall the subgroup $J_\Sigma$ from \S\ref{sub:notation yu}. Let $\gamma\in G_{\rs}$, $g\in G$, and suppose that $\gamma\in\, H_{\Sigma,g}:=\,^g\!J_\Sigma\cap G_{\xo}$. Let
\[
\aconst_{\gamma,\Sigma}:=
\left\{\begin{array}{ll} h_G\cdot\sd(\gamma)+s_{d-1}\quad &\textrm{if }\ h_G>1,\\ \\
\sd(\gamma)&\textrm{if }\ h_G=1.\end{array}\right.
\]
Then,  we have
\begin{enumerate}
\item[(i)]
$T^\gamma_{\aconst_{\gamma,\Sigma}^+}\subset \,H_{\Sigma,g}$,
\item[(ii)]
$\sharp\left(\left.(T^\gamma\cap G_{\xo})\,H_{\Sigma,g}\right/\,H_{\Sigma,g}\right)\le  q^{\abrank_G(\aconst_{\gamma,\Sigma}+1)}$.
\end{enumerate}
\end{prop}

\begin{proof}
For simplicity of notation, we write $\aconst$ for $\aconst_{\gamma,\Sigma}$.

(i) By Lemma \ref{lem: centralizer1}, $T^\gamma_{\aconst^+}\subset G_{x}$.
Since $T^\gamma_{\aconst^+}=\bT^\gamma(E)_{\aconst^+}\cap G_{0^+}\subset\,^g\!K_E\cap G_{0^+}$ from Lemmas \ref{lem: centralizer2} and \ref{lem: centralizer2-rank1}, we have $T^\gamma_{\sd(\gamma)^+}\subset\, ^g\!J_\Sigma\cap G_{0^+}$. Hence, $T^\gamma_{\aconst^+}\subset H_{\Sigma,g}$.

\smallskip

(ii) Note that $G_{\xo}\cap T^\gamma\subset T^\gamma_0$. Then, we have
\[
\sharp\left(\left.(T^\gamma\cap G_{\xo})\,H_{\Sigma,g})\right/\,H_{\Sigma,g}\right)\le\sharp \left(\left.T^\gamma_0\right/T^\gamma_{\aconst^+}\right)\le q^{\abrank_G(\aconst+1)}.
\]
\end{proof}

\subsection {Inverse image under conjugation}\label{sub:inv-image} In this subsection we prove a lemma to control the volume change of an open compact subgroup under the conjugation map as we will need the result in Proposition \ref{prop: main} below. For regular semisimple elements $g\in G$ and $X\in \fkg$, denote by $\fkg_g$ and $\fkg_X$ the centralizer of $g$ and $X$ in $\fkg$, respectively. Define the Weyl discriminant as
  \[D(g):=|\det(\mathrm{Ad}(g)|\fkg/\fkg_g)|,\quad D(X):=|\det(\mathrm{ad}(X)|\fkg/\fkg_X)|.\]
Recall that the rank of $G$, to be denoted $\abrank_G$, is the dimension of (any) maximal torus in $G$. Define $\psi(\abrank_G)$ to be the maximal $d\in \Z_{\ge 0}$ such that $\phi(d)\le \abrank_G$, where $\phi$ is the Euler phi-function. Put $N(G):=\max(\psi(\abrank_G),\dim G)$.
Under the assumption that $p>N(G)+1$, recall from \cite[App B, Prop B]{Wal08} that there is a homeomorphism
  \[\exp:\fkg_{0^+}\stackrel{\sim}{\ra} G_{0^+}.\]
  Under the hypothesis $\Hypk$ this $\exp$ map is filtration preserving, in particular, it preserves the ratio of volumes.

\begin{lem}\label{l:conj-inv-estimate}
Suppose  $p>N(G)+1$ and $\Hypk$ is valid. Let $x\in \cB(\bG)$. Let $\gamma\in G_{x}\cap G_{0^+}$ and suppose that $\gamma$ is regular semisimple (so that $D(\gamma)\neq 0$). Consider the conjugation map $\psi_\gamma:G\ra G$ given by $\delta\mapsto \delta \gamma\delta^{-1}$. For each open compact subgroup $H\subset G_{\xo}$ containing $\gamma$, we have
  \[\frac{\vol_{G/Z}(\psi_\gamma^{-1}(H)\cap G_{\xo})}{\vol_{G/Z}(H)}\le \sharp\left((G_{\xo}\cap T^\gamma)H/H\right)\cdot  D(\gamma)^{-1}\cdot\sharp\left(H/(H\cap G_{x,0^+})\right).\]
\end{lem}

\begin{proof}
It suffices to prove that the left hand side is bounded by $\sharp\left(G_{\xo}\cap T^\gamma/H\cap T^\gamma \right)\cdot  D(\gamma)^{-1}$ under the additional assumption that $H\subset G_{\xo}\cap G^+_0$. Indeed, in general, one only needs to also count the contribution on $H/(H\cap G_{x,0^+})$, which is bounded by its cardinality.

  Put $Y:=\exp^{-1}(\gamma)$ and $\fkh:=\exp^{-1}(H)\subset \fkg_{x,0^+}$.
  Note that $\fkh$ is an $\cO_k$-lattice in $\fkg$ since $H$ is an open compact subgroup.
  Write $C_Y:\fkg\ra \fkg$ for the map $X\mapsto [X,Y]$, whose restriction to $\fkg_{0^+}$ is going to be denoted by $C_{Y,0^+}$.
  Define $C_\gamma:G_{0^+}\ra G_{0^+}\gamma^{-1}$ by $\delta\mapsto \delta\gamma\delta^{-1}\gamma^{-1}$, which is the composition of $\psi_\gamma$ with the right multiplication by $\gamma^{-1}$. Since $\gamma\in H$ we have $C_\gamma^{-1}(H)=\psi_\gamma^{-1}(H)\subset G_{0^+}$. Via the exponential map, $C_\gamma:\psi_\gamma^{-1}(H)\cap G_{\xo}\ra H$ corresponds to
  \[C_{Y,0^+}:\exp^{-1}(C^{-1}_\gamma(H)\cap G_{\xo})=C_{Y,0^+}^{-1}(\fkh\cap \frg_{\xo})\ra \fkh.\]
  Since the exponential map preserves the ratio of volumes,
  \[\frac{\vol_{G/Z}(\psi_\gamma^{-1}(H)\cap G_{x})}{\vol_{G/Z}(H)} = \frac{\vol_{G/Z}(C_{Y,0^+}^{-1}(\fkh)\cap \frg_{\xo})}{\vol_{G/Z}(\fkh)}
  \le \frac{\vol_{G/Z}(C_{Y}^{-1}(\fkh)\cap \frg_{\xo}))}{\vol_{G/Z}(\fkh)} = [C_{Y}^{-1}(\fkh)\cap \frg_{\xo}:\fkh] .\]

  We will be done if we show that $[C_{Y}^{-1}(\fkh)\cap \frg_{\xo}:\fkh]\le D(Y)^{-1} [\fkg_Y\cap \frg_{\xo}: \fkg_Y\cap \fkh]$. Write $\mathrm{pr}$ for the projection map $\frg_{\xo} \ra \frg_{\xo}/(\frg_{\xo}\cap \fkg_Y)$. Consider the commutative diagram
\[\xymatrix{ \fkg_Y \cap \frg_{\xo} \ar[r] \ar@{^(->}[d] & \fkg_Y \cap \fkh \ar@{^(->}[r] \ar@{^(->}[d] & \fkg_Y \cap \frg_{\xo}  \ar@{^(->}[d] \\
  C_Y^{-1}(\fkh) \cap \frg_{\xo} \ar[r]^-{C_Y} \ar@{^(->}[d] &   \fkh \ar@{^(->}[r] \ar@{^(->}[d] &  \frg_{\xo}  \ar[d]^-{\mathrm{pr}} \\
  \mathrm{pr}(C_Y^{-1}(\fkh)\cap \frg_{\xo})  \ar[r]^-{\ol C_Y} &  \mathrm{pr}( \fkh ) \ar@{^(->}[r]&  \frg_{\xo} /(\fkg_Y\cap \frg_{\xo}) ~\subset~ \fkg/\fkg_Y}.\]
  Denote by the preimage of $\mathrm{pr}(\fkh)$ in $\mathrm{pr}(C_Y^{-1}(\fkh)$ (resp. $C_Y^{-1}(\fkh) \cap \frg_{\xo}$) by $L_2$ and $L_1$. Then $L_2$ and $L_1$ are $\cO_k$-lattices in $\fkg/\fkg_Y$ and $\fkg$, respectively, where $\cO_k$ is the ring of integers in $k$. So we have
  \[[C_{Y}^{-1}(\fkh)\cap \frg_{\xo}:\fkh]\le [L_1:\fkh] = [\fkg_Y\cap \frg_{\xo}: \fkg_Y\cap \fkh] [L_2:\mathrm{pr}(\fkh)].\]
  Since $\ol{C}_Y=\mathrm{ad}(Y)$ as $k$-linear isomorphisms on $\fkg/\fkg_Y$, we see that $[L_2:\mathrm{pr}(\fkh)] = D(Y)^{-1}$, completing the proof.
\end{proof}

Combining Proposition \ref{prop: centralizer0} and Lemma \ref{l:conj-inv-estimate}, we have the following:

\begin{cor}\label{cor:conj-inv-estimate}
We keep the situation and notation from Proposition \ref{prop: centralizer0}. Then, we have
\[
\frac{\vol_{G/Z}(\psi_\gamma^{-1}(\,H_{\Sigma,g})\cap G_{\xo})}{\vol_{G/Z}(\,H_{\Sigma,g})}
\le D(\gamma)^{-1} q^{\dim(G)+\abrank_G(\aconst_{\gamma,\Sigma}+1)}
\]
\end{cor}

\subsection{Intersection of $L_{\stm}$ with a maximal unipotent subgroup}
For later use, we study the intersections of $L_{\stm}$ with unipotent subgroups in this subsection.
Consider a tamely ramified twisted Levi sequence $(\bG',\bG)$.
  Let $\bT$ (resp. $\bT'$) be a maximally $k$-split maximal torus of $\bG$ (resp. $\bG'$) such that $\bT^{\prime s}\subset\bT^s$ where $\bT^s$ and $\bT^{\prime s}$ are the $k$-split components of $\bT$ and $\bT'$ respectively. Set $M:=Z_G^\circ(T^s)$ (resp. $M':=Z_G^\circ(T^{\prime s})$), a minimal Levi subgroup of $G$ containing $T$ (resp. $T'$). Note that $M\subset M\rq$. Fix a parabolic subgroup $MN$ of $G$ (resp. $M'N'$), where $N$ (resp. $N'$) is the unipotent subgroup such that $N'\subset N$.

\begin{lem} \label{lem: inequality} We keep the notation from above. Let $x\in\Apt(T')$.

\begin{itemize}
\item[(i)]
For any $\rma,\rma'\in \mathbb R$ with $\rma>\rma'>0$,
\[
[(G_{x,\rma'}\cap N'):((G'_{x,\rma'}G_{x,\rma})\cap N')]^2\le  q^{\dim_k(N)}[G_{x,\rma'}:(G'_{x,\rma'}G_{x,a})].
\]

\item[(ii)]
\[
[(G_{x}\cap N'):((G'_{x}G_{x,0^+})\cap N')]^2\le  q^{\dim_k(N)}[G_{x}:(G'_{x',0}G_{x,0^+})].
\]

\item[(iii)]
\[
[(G_{\xo}\cap N'):((G'_{x}G_{x,\rma})\cap N')]^2 \le q^{2\dim_k(N)}[G_{\xo}:(G'_{x}G_{x,\rma})].
\]
\end{itemize}
\end{lem}

\proof
(i) Both $\CaJ:=G_{x,\rma'}$ and $\CaJ':=(G'_{x,\rma'}G_{x,\rma})$ are decomposable with respect to $M'$ and $N'$, that is,  $\CaJ=(\CaJ\cap \overline N')\cdot(\CaJ\cap M')\cdot (\CaJ\cap N')$, etc. Write $Y_{X}:=Y\cap X$  for any $X, Y\subset G$. Then,
$[\CaJ:\CaJ']=[\CaJ_{\overline N'}:\CaJ'_{\overline N'}]\cdot[\CaJ_{M'}:\CaJ'_{M'}]\cdot [\CaJ_{N'}:\CaJ'_{N'}]$
and we have
\begin{eqnarray}
[\CaJ:\CaJ']&=&[\CaJ_{M'}:\CaJ'_{M'}][\CaJ_{N'}:\CaJ'_{N'}][\CaJ_{\overline N'}:\CaJ'_{\overline N'}]
\ge[\CaJ_{N'}:\CaJ'_{N'}][\CaJ_{\overline N'}:\CaJ'_{\overline N'}]\nonumber\\
&\ge& \frac1{q^{\dim_k(N')}}[\CaJ_{N'}:\CaJ'_{N'}]^2
\ge \frac1{q^{\dim_k(N)}}[\CaJ_{N'}:\CaJ'_{N'}]^2 .\nonumber
\end{eqnarray}

(ii) This follows from
\begin{eqnarray}
[G_{x}:G_{x}'G_{x+}]&\ge&[(G_{x})_{N'}G_{x,0^+}:(G_{x}'G_{x+})_{N'}G_{x,0^+}]\cdot[(G_{x})_{\overline N'}G_{x,0^+}:(G_{x}'G_{x+})_{\overline N'}G_{x,0^+}]\nonumber\\
&=&[(G_{x})_{N'}G_{x,0^+}:(G_{x}'G_{x+})_{N'}G_{x,0^+}]\cdot[(G_{x})_{\overline N'}:(G_{x}'G_{x+})_{\overline N'}]\nonumber\\
&\ge&\frac1{q^{\dim_k(N)}}[(G_{\xo})_{N'}:(G'_{x}G_{x,0^+})_{N'}]^2.\nonumber
\end{eqnarray}

\smallskip

(iii) We have
\begin{eqnarray}
[G_{\xo}:(G'_{x}G_{x,\rma})]&=&[G_{\xo}:(G'_{x}G_{x,0^+})][(G'_{\xo}G_{\xo,0^+}):(G'_{x}G_{x,\rma})]\nonumber\\
&=&[G_{\xo}:(G'_{x}G_{x,0^+})][G_{\xo,0^+}:(G'_{x,0^+}G_{x,\rma})].\nonumber
\end{eqnarray}
Combining (i) and (ii), we have
\begin{eqnarray}
[G_{\xo}:(G'_{x}G_{x,\rma})]&\ge& \frac1{q^{2\dim_k(N)}}[(G_{\xo})_{N'}:(G'_{x}G_{x,0^+})_{N'}]^2[(G_{\xo,0^+})_{N'}:(G'_{x,0^+}G_{x,\rma})_{N'}]^2
\nonumber\\
&=&\frac1{q^{2\dim_k(N)}}[(G_{\xo})_{N'}:(G'_{x}G_{x,\rma})_{N'}]^2.\nonumber
\end{eqnarray}
\qed

\section{Asymptotic behavior of supercuspidal characters}

\subsection{Main theorem}\label{sub:main-thm}

\begin{conj}\label{c:asymptotic-char}
 Consider the set of $\pi$ in $\Irr^{\scusp}(G)$ such that the central character of $\pi$ is unitary.
   For each fixed $\gamma\in G_{\rs}$,
\[ \frac{\Theta_{\pi}(\gamma)}{\fdeg(\pi)}\ra 0\quad \mbox{as}~\deg(\pi)\ra \infty;\]
   namely for each $\epsilon>0$ there exists $d_\epsilon>0$ such that $|\Theta_{\pi}(\gamma)/\fdeg(\pi)|<\epsilon$ whenever $\deg(\pi)>d_\epsilon$.
  \end{conj}

  Our main theorem in the qualitative form is a partial confirmation of the conjecture under the hypotheses discussed in \S\ref{sub:hypo} and above Lemma \ref{l:conj-inv-estimate}. 

  \begin{thm}\label{t:asymptotic-char} Suppose that $\Hypk$, $\HypB$ and $\HypGT$ are valid. Then, Conjecture \ref{c:asymptotic-char} holds true if $\gamma$ and $\pi$ are restricted to the sets $G_{0^+}$ and $\Irr^{\Yu}(G)$, respectively. 
\end{thm}

  The proof is postponed to \S\ref{sub:proof-main-thm} below. Actually we will establish a rather explicit upper bound on $|\Theta_{\pi}(\gamma)/\fdeg(\pi)|$, which will lead to a quantitative strengthening of the above theorem. See Theorem \ref{thm:uniform-bound} below.

  The central character should be unitary in the conjecture; it would be a problem if each $\pi$ is twisted by arbitrary (non-unitary) unramified characters of $G$. However the assumption plays no role in the theorem since every unramified character is trivial on $G_{0^+}$.

 We are cautious to restrict the conjecture to supercuspidal representations as our result does not extend beyond the supercuspidal case. However it is natural to ask whether the conjecture is still true for discrete series representations. As a small piece of psychological evidence we verify the analogue of Conjecture \ref{c:asymptotic-char} for discrete series of real groups on elliptic regular elements in \S\ref{sub:real-group-char} below.

\subsection{Reductions} \label{rmk: reductions} \rm Let $\pi$ be as in Theorem \ref{t:asymptotic-char} associated to a generic $G$-datum $\Sigma$.
In proving Theorem \ref{t:asymptotic-char}, we may assume that each $\pi$ (and hence $\Sigma$) is associated to a fixed orbit of $(\vec\bG,\xo)$ since there are only a finite number of $(\vec\bG,\xo)$ up to conjugacy.
Let $\bT^\gamma$ denote the unique maximal torus containing $\gamma$, and pick any  $y\in\Apt(T^\gamma)=\Bd(T^\gamma)$. We may assume the following without loss of generality.

\begin{enumerate}
\item[(1)] $\phi_d=1$: since $\Theta_{\overline\phi_d\otimes\pi}(\gamma)=\overline\phi_d(\gamma)\Theta_{\pi}(\gamma)$ and $\deg(\pi)=\deg(\overline\phi_d\otimes\pi)$, it is enough to verify the theorem only for the generic $G$-data such that $\phi_d=1$. Note that  the depth of $\pi$ is given by $\rtm_{d-1}$ if $\phi_d=1$.

\item[(2)] $\rtm_{\Sigma}=\rtm_{d-1}\ge 1$: Lemma \ref{l:bound-on-d} implies that $\fdeg(\pi)\rightarrow\infty$ is equivalent to $\vol_{G/Z}(J_{\Sigma})\rightarrow0$ as $i\rightarrow\infty$, which is in turn equivalent to $r_{\Sigma}\rightarrow\infty$. Hence, we may assume $\rtm_{\Sigma}\ge 1$ without loss of generality.

\item[(3)] $x\in\overline C_y$ where $C_y$ is a facet of maximal dimension in $\Bd(G)$ containing $y$ in its closure: this is due to that the $G$-orbit of $x$ in $\Bd(G)$ intersects with $\overline C_y$ nontrivially.

\end{enumerate}
The following is a consequence of (3):

\begin{enumerate}

\item[(3)$^\prime$] $\gamma\in G_{x}$ since $\gamma\in T^\gamma_{0^+}\subset G_{y,0^+}\subset G_x$. Moreover, $G_{y,\rtm}\subset G_x$ for any $r>0$ and $G_{y,r^+}\supset G_{x,(r+1)^+}\supset G_{x,\urtm^+}$. where $\urtm:=\lceil r+1\rceil\in\Z$.

\end{enumerate}

\subsection{Mackey's theorem for compact induction}\label{sub:Mackey} Besides the local constancy of characters, we are going to need the classical Mackey's theorem in the context of compactly induced representations.

\begin{lem}\label{l:Mackey}
  Let $J\subset G$ be an open compact mod center subgroup and $H\subset G$ a closed subgroup. Let $(J,\rho)$ and $(H,\tau)$ be smooth representations such that $\dim \rho<\infty$. Then
  \[\Hom_G(\cind^G_J \rho,\ind^G_H \tau)\simeq \bigoplus_{g\in H\backslash G/J} \Hom_{J\cap H^g}(\rho,\tau^g).\]
\end{lem}

  In fact it is a canonical isomorphism. A natural map will be constructed in the proof below.

\begin{proof}
  Since the details are in \cite{Kut77}, where a more general result is proved, we content ourselves with outlining the argument. Let $S(\rho,\tau)$ denote the space of functions $s:G\ra \End_{\C}(\rho,\tau)$ such that $s(hgj)=\tau(h)s(g)\rho(j)$ for all $h\in H$, $g\in G$, and $j\in J$. For each $g\in H\backslash G/J$ define $S_{\xo}(\rho,\tau)$ to be the subspace of $s\in S(\rho,\tau)$ such that $\supp s\subset HxJ$. Clearly $S(\rho,\tau)=\oplus_{g\in H\backslash G/J} S_{g}(\rho,\tau)$. For each $v\in \rho$ we associate $f_v\in S(\rho,\tau)$ such that $f_v(j)=\rho(j)v$ if $j\in J$ and $f_v(j)=0$ if $j\notin J$. We define a map
  \[\Hom_G(\cind^G_J \rho,\ind^G_H \tau)\ra S(\rho,\tau),\quad \phi\mapsto s_\phi\]
  such that $s_\phi(g)$ sends $v\in \rho$ to $(\phi(f_v))(g)$. We also have a map
  \[S_{g}(\rho,\tau)\ra \Hom_{J\cap H^g}(\rho,\tau^g),\quad s\mapsto s(g).\]
  (It is readily checked that $s(g)\in \Hom_{J\cap H^g}(\rho,\tau^g)$.) It is routine to check that the two displayed maps are isomorphisms.
\end{proof}

The following corollary was observed in \cite[Lem 4.1]{Nev13}.

\begin{cor}\label{c:Mackey}
  In the setup of Lemma \ref{l:Mackey}, further assume that $\cind^G_J \rho$ is admissible and that $H$ is an open compact subgroup. Then
  \[\left(\cind^G_J \rho\right)|_H \simeq \bigoplus_{g\in J\backslash G/H} \cind^H_{H\cap J^g } \rho^g.\]
\end{cor}

\begin{proof}
  Let $(H,\tau)$ be an admissible representation. Then by Frobenius reciprociy and the preceding lemma,
  \[\Hom_H(\cind^G_J \rho,\tau)\simeq \Hom_G(\cind^G_J \rho,\ind^G_H \tau)\simeq\bigoplus_{g\in H\backslash G/J} \Hom_{J\cap H^g}(\rho,\tau^g).\]
  Further, through conjugation by $g$ and Frobenius reciprocity, the summand is isomorphic to
  \[\Hom_{J^{g^{-1}}\cap H}(\rho^{g^{-1}},\tau)
  \simeq \Hom_{H}\left(\cind^H_{J^{g^{-1}}\cap H}\rho^{g^{-1}},\tau\right).\]
  The proof is finished by replacing $g$ by $g^{-1}$ in the sum and applying Yoneda's lemma.
\end{proof}

\subsection{Main estimates} This subsection establishes the main estimates towards of the proof of Theorem \ref{t:asymptotic-char}. Let $\Sigma$ be a generic $G$-datum associated to $\pi$. That is, $\pi\simeq\pi_{\Sigma}$. In this entire subsection we keep conditions (1), (2) and (3) of \S\ref{rmk: reductions}.

Lemma \ref{l:bound-on-d} implies that $\fdeg(\pi)\rightarrow\infty$ is equivalent to $\vol_{G/Z}(J_{\Sigma})\rightarrow0$, which is in turn equivalent to $r_{\Sigma}\rightarrow\infty$. Henceforth we will often drop the subscript $\Sigma$ when the context is clear.

\begin{lem}\label{lem: loc const}
Let $\Sigma$, $\pi$, and $\gamma$ be as above. For simplicity, we write $r$ for $\rtm_\datum$.
Suppose
\[
\sd(\gamma)\le \frac r2.
\]
Then we have
\begin{equation}\label{eq: loc const}
\Theta_\pi(\gamma)=\mathrm{Tr}\left(\pi(\gamma)\left|V_\pi^{G_{y,r^+}}\right.\right).
\end{equation}
for any $y\in\Apt(T^\gamma)=\Bd(T^\gamma)$. Here $V_\pi^{G_{y,r^+}}$ is the space of $G_{y,r^+}$-invariants in $V_\pi$.
\end{lem}

\begin{remark}
Recall by \S\ref{rmk: reductions}-(2), the depth of $\pi_\datum$ is $r=r_\datum$. Note that $\gamma \in G_{[y]}$. Hence, $\gamma$ normalizes $G_{y,r^+}$, and the right-hand side of the formula in \eqref{eq: loc const} is well-defined.
\end{remark}

\begin{proof}
Given a subset $X\subset G$, let $\ch_{X}$ denote the characteristic function of $X$.
By \cite[Cor 12.9]{AK07} (see also \cite{Meyer-Solleveld:growth}) $\Theta_\pi$ is constant on $\gamma G_{y,r^+}\subset \,^{G_{y,0^+}}\!\left(\gamma T^\gamma_{\rtm^+}\right)$.\footnote{Since $\gamma$ is regular, the summation in \cite[Cor 12.9]{AK07} runs over no nilpotent elements other than $0$. So the corollary tells us that $\Theta_\pi(\gamma')$ is equal to a constant $c_0$ for all $\gamma'$ in the $G$-conjugacy orbit of $\gamma+T^\gamma_a$ for $a>\max(2\sd(\gamma),\rho(\pi))$, where $\rho(\pi)$ denotes the depth of $\pi$, which is $\rtm$. For $y\in \Apt(T^\gamma)$, we have $G_{y,a}$ contained in the $G$-orbit of $\gamma+T_a$. In our case $\max(2\sd(\gamma),\rho(\pi))=r$, thus $\Theta_\pi$ is indeed constant on $\gamma G_{y,r^+}$.} Thus we have
\[
\begin{array}{ll}
\Theta_\pi(\gamma)&=\frac1{\vol_{G}(G_{y,r^+})}\int_G\Theta_\pi(g)\ch_{\gamma G_{y,r^+}}dg \\
&=\frac1{\vol_{G}(G_{y,r^+})}\mathrm{Tr}(\pi(\ch_{\gamma G_{y,r^+}}))
=\mathrm{Tr}\left(\pi(\gamma)\pi\left(\frac{\ch_{G_{y,r^+}}}{\vol_{G}(G_{y,r^+})}\right)\right)\\
&=\mathrm{Tr}\left(\pi(\gamma)|V_\pi^{G_{y,r^+}}\right).
\end{array}
\]
The last equality follows from the fact that $\pi
\left(\frac{\ch_{G_{y,r^+}}}{\vol_{G}(G_{y,r^+})}\right)$ is the projection of $V_\pi$ onto $V_\pi^{G_{y,r^+}}$.
\end{proof}

Our aim is to prove Proposition \ref{prop: main} below using Lemma \ref{lem: loc const}. Recall $\yo\in \Apt(T^\gamma)$ is fixed. If $V_\pi^{G_{y,r^+}}=0$, we have $\Theta_\pi(\gamma)=0$. Hence, from now on, we assume that $V_\pi^{G_{y,r^+}}\neq0$ without loss of generality.
In the following series of lemmas, we first describe the space $V_\pi^{G_{y,r^+}}$.
The following result is originally due to Jacquet \cite{Jac71}.

\begin{lem}\label{lem:Mackey2}  Let $J$ be an open compact mod center subgroup of $G$ and $\rho$ an irreducible representation of $J$ such that $\pi=\cind_{J}^G\rho$ is irreducible (thus supercuspidal). Then, for any nontrivial unipotent subgroup $N$ of $G$, we have $V_\rho^{N\cap J}=0$.
\end{lem}

\proof
Applying Frobenius reciprocity and Lemma \ref{l:Mackey} with $H=N$,
\[
0=\Hom_N(\pi,1_N)=\Hom_G(\cind_{J}^G \rho, \Ind_N^G1_N)
\simeq \oplus_{g\in N\backslash G/J}\Hom_{J\cap N^g}(\rho, 1_{N^g}).
\]
\qed

Let $J:=J_{\Sigma}$ and $\rho:=\rho_\Sigma$. We deduce from Corollary~\ref{c:Mackey} that
\begin{equation}\label{Vpi-dec}
\mathrm{Res}_{G_{\xo}}\cind_J^G\rho\simeq\,\bigoplus_{g\in G_{\xo}\backslash G/J} \Ind_{G_{\xo}\cap \,^g\!J}^{G_{\xo}}\,^g\!\rho.
\end{equation}

\begin{defrmk} \

\begin{enumerate}
\item[(1)]
Define
\[\cX_\Sigma':=\left\{g\in G\mid G_{g^{-1}\xo, \urtm^+}\cap N\supset G_{\xo}\cap N\textrm{ for some unipotent subgroup }N\neq \{1\}\right\}\]
and
\[
\cX_\Sigma:=G-\cX_\Sigma'.
\]
We observe that
\begin{enumerate}
\item $G_{\xo}\cap N=G_{[\xo]}\cap N\supset J\cap N$ for any unipotent subgroup $N$, and
\item $\cX_\Sigma'$, $\cX_\Sigma$ are left and right $G_{[\xo]}$-invariant.
\end{enumerate}
\item[(2)]
Suppose $\rtm$ is sufficiently large so that $G_{\yo,\rtm^+}\subset G_{\xo}$. Set
\[
\cX_\Sigma^\circ:=\left\{g\in G\left| \left(\Ind_{G_{\xo}\cap \,^g\!J}^{G_{\xo}}\,^g\!\rho\right)^{G_{\yo,\rtm^+}}\neq 0\right.\right\}.
\]
In this case, from Lemma \ref{lem: loc const} and the above definition it is clear that
\begin{equation}\label{e:loc-const-refined}
\Theta_\pi(\gamma)=\sum_{g\in G_{\xo}\left\backslash\cX_\Sigma^\circ\right/J}\Tr\left(\pi(\gamma)\left| \left(\Ind_{G_{\xo}\cap \,^g\!J}^{G_{\xo}}\,^g\!\rho\right)^{G_{\yo,\rtm^+}}\right.\right).
\end{equation}
Note that $\cX_\Sigma^\circ$ is right $J$-invariant, and also left $G_{\xo}$-invariant.
\end{enumerate}
\end{defrmk}

\begin{lem}\label{l:no-inv-in-ind} 
	If $g\in \cX'_\Sigma$, then the space $\Ind_{G_{\xo}\cap\,^g\!J}^{G_{\xo}} {}^g\!\rho$ has no nonzero $G_{\yo,r^+}$-invariant vector. That is, $\cX_\Sigma^\circ\cap \cX'_\Sigma=\emptyset$, or equivalently $\cX_\Sigma^\circ\subset\cX_\Sigma$.
\end{lem}


\begin{proof}
Recall that $G_{y,r}\subset G_x$ (see \S\ref{rmk: reductions}) and hence $G_{y,r^+}\subset G_x$.
Another application of~Mackey's formula yields
\[
\mathrm{Res}_{G_{\yo,r^+}}\Ind_{G_{\xo}\cap \,^g\!J}^{G_{\xo}}\,^g\!\rho\simeq\bigoplus_{h\in G_{\yo,r^+}\backslash G_{\xo}/G_{\xo}\cap \, ^{g}\!J} \Ind_{G_{\yo,r^+}\cap \, ^{hg}\!J}^{G_{\yo,r^+}}\, ^{hg}\!\rho.
\]
(This is derived from the formula in representation theory of finite groups since $[G_{\xo}:G_{\xo}\cap \,gJg^{-1}]$ and $\dim \rho$ are finite. We do not need Corollary \ref{c:Mackey}.) By Frobenius reciprocity and conjugation by $hg$, we obtain
\[\Hom_{G_{\yo,r^+}}\left(1,\Ind_{G_{\yo,r^+}\cap \, ^{hg}\!J}^{G_{\yo,r^+}}\, ^{hg}\!\rho\right)\simeq \Hom_{G_{\yo,r^+}\cap \, ^{hg}\! J}\left(1,\, ^{hg}\!\rho\right)\simeq
\Hom_{G_{g^{-1}h^{-1}y,r^+}\cap J}(1,\rho).\]
Since $G_{\yo,\rtm^+}\supset G_{\xo,\urtm^+}$ and $h\in G_{\xo}$, we have $G_{h^{-1}\yo,\rtm^+}\supset
G_{h^{-1}\xo,\urtm^+}=G_{\xo,\urtm^+}$ and thus $G_{g^{-1}h^{-1}\yo,\rtm^+}\supset
G_{g^{-1}\xo,\urtm^+}$.  It suffices to verify that the last Hom space is zero. If $g\in  \cX_\Sigma'$,  then $G_{g^{-1}\xo,\urtm^+}\cap N \supset J\cap N$ for some $N$ and thus $G_{g^{-1}h^{-1}\yo,\rtm^+}\cap J \supset J\cap N$. This and Lemma~\ref{lem:Mackey2} imply that the Hom space indeed vanishes.
\end{proof}


For simplicity, we will write $\cX^\circ$, $\cX$ and $\cX'$ for $\cX_\Sigma^\circ$, $\cX_\Sigma$ and $\cX'_\Sigma$ when the context is clear. For the purpose of our character computation, it is natural to estimate the cardinality of $G_{x}\bs \cX^\circ_\Sigma/J$ in view of \eqref{e:loc-const-refined}. Instead we bound the size of $G_{x}\bs \cX_\Sigma/J$, which is larger by the preceding lemma but easier to control.
To this end we begin by setting up some notation for the Cartan and Iwahori decompositions.

\begin{numbering}{\it Notation.}\rm \
Let $\mathbf T^0\subset\mathbf G^0$ and $\mathbf T\subset\mathbf G$ be maximal and maximally $k$-split tori such that $x\in\mathcal A(\mathbf T^0, k)\subset\mathcal A(T):=\mathcal A(\mathbf T, \mathbf G, k)\subset \Bd(G)$. Let $C$ be a facet of maximal dimension in $\mathcal A(T)$ with $x\in \overline C$. Let $\Delta$ be the set of simple $T$-roots associated to $C$, and $N_\Delta$ the maximal unipotent subgroup with simple roots $\Delta$.
Let $G_{C}\subset G_{\xo}$ be the Iwahori subgroup fixing $C$ so that $G=G_{C} N_G(T)G_{C}$ and $G_{z}$ be a special maximal parahoric subgroup with $z\in\overline C$ such that $G_{C}\subset G_{z}$. Note that elements in $W:=N_G(T)/C_G(T)$ can be lifted to elements in $G_{z}$. Let $T^-:=\{t\in T\mid  tUt^{-1}\subset U \textrm{ for any open subgroup}~U~\textrm{in }N_\Delta\}$  so that we have a Cartan decomposition $G=G_{z}T^-G_{z}$.

\end{numbering}

\begin{lem}\label{lem: X_0-1} 
Let $T^-(\urtm):=\{t\in T^-\mid 1\le|\alpha(t^{-1})|\le q^{\urtm+2},\ \alpha\in\Delta\}$. Then
\begin{equation}\label{pfGzX}
	\cX^\circ\subset\cX \subset G_{z} T^-(\urtm) G_{z}
\end{equation}
\end{lem}

\proof
The first inclusion is from Lemma \ref{l:no-inv-in-ind}. For the second inclusion,
for $v\in\Bd(G)$ and $a\in\bbR_{\ge0}$, let
\begin{align*}
\cX'(v,a)&=\left\{g\in G\mid G_{g^{-1}v, a^+}\cap N\supset G_{v}\cap N\textrm{ for some unipotent subgroup }N\neq \{1\}\right\} \\
\cX(v,a)&=G-\cX'(v,a).
\end{align*}

It is enough to show that
\begin{equation}\label{eq:inclusions}
\left(G-G_{z}T^-(\urtm)G_{z}\right)\subset\cX'(z,\urtm+1)\subset\cX'(x,\urtm)=\cX'.
\end{equation}

For the second inclusion in $(\ref{eq:inclusions})$, let $g\in \cX'(z,\urtm+1)$. Since $\xo,z\in\overline C$ and $G_{g^{-1}z,(\urtm+1)^+}\cap N\supset G_{z}\cap N$, we have $G_{g^{-1}x,\urtm^+}\supset G_{g^{-1}z,(\urtm +1)^+}\supset G_{g^{-1}z,(\urtm+1)^+}\cap N\supset G_{z}\cap N\supset G_{\xo}\cap N$. Hence, $g\in\cX'(x,\urtm)$.

For the first inclusion in $(\ref{eq:inclusions})$, let $g=g_1t'g_2\in G-G_{z}T^-(\urtm)G_{z}$ with $g_i\in G_{z}$ and $t'\in T-T^-(\urtm)$. Then, there is $\alpha\in \Delta$ such that $|\alpha(t^{\prime-1})|>q^{\urtm+2}$. Let $N_\alpha$ be the maximal unipotent subgroup associated to $\alpha$. Then,
\[G_{t^{\prime-1}z, (\urtm+1)^+}\cap N_\alpha\supset G_{z}\cap N_\alpha.\]
Since we have $G_{g^{-1}z,(\urtm+1)^+}=G_{g_2^{-1}t^{\prime-1}z,(\urtm+1)^+}={}\,^{g_2^{-1}}G_{t^{\prime-1}z, (\urtm+1)^+}$ and $G_{g_2 z}=G_{z}$, we have \[G_{g^{-1}z,(\urtm+1)^+}\cap \,^{g_2^{-1}}\!N_\alpha\supset G_{z}\cap \,^{g_2^{-1}}\!N_\alpha, \quad\mbox{thus}\quad G_{t^{\prime-1}z,(\urtm+1)^+} \cap N_\alpha \supset G_{z} \cap N_\alpha.\]
We conclude that $g\in \cX'(z,\urtm+1)$.
\qed

\

To give another description of $\cX^\circ$, we define compact mod center sets $S_{\xo,\yo} \subset \mathcal S_{\xo,\yo}$ as follows:
\[
S_{\xo,\yo}:=G_{[\yo]}S_{\xo} G_{\yo,0^+},\qquad \mathcal S_{\xo,\yo}:=G_{x}S_{\xo,\yo} G_{\xo}.
\]In particular the quotient $(Z_GG_{\xo})\backslash \mathcal S_{\xo,\yo}$ is finite.
Recall that $S_{\xo}$ is the set constructed in the proof of Lemma \ref{lem: good type}.

\begin{lem}\label{lem: X_0-2} Suppose $\HypB$ and $\HypGT$ are valid.  Then, for any double coset $G_{x} gJ\subset \cX^\circ$, we have
\[
G_{x} g G_{x}\cap S_{\xo,\yo} \,G^{d-1} G_{\xo,0^+}\neq \emptyset.
\]
\end{lem}

\proof Since $G_{\yo,\rtm}\subset G_{\xo}$ by \S\ref{rmk: reductions},
Mackey's formula gives us, as in the proof of Lemma \ref{l:no-inv-in-ind}, that
\[
\mathrm{Res}_{G_{\yo,r}}\Ind_{G_{\xo}\cap \,^g\!J}^{G_{\xo}}\,^g\!\rho\simeq\bigoplus_{\ell\in G_{\yo,r}\backslash G_{\xo}/G_{\xo}\cap \, ^{g}\!J} \Ind_{G_{\yo,r}\cap \, ^{\ell g}\!J}^{G_{\yo,r}}\, ^{\ell g}\!\rho,
\]
Since $g\in\cX^\circ$ there is $\ell\in G_{x}$ such that
$\left( \Ind_{G_{\yo,r}\cap \, ^{\ell g}\!J}^{G_{\yo,r}}\, ^{\ell g}\!\rho\right)^{G_{\yo,\rtm^+}}\neq 0$. By replacing $g$ with $\ell^{-1} g$ if necessary, we may assume $\left( \Ind_{G_{\yo,r}\cap \, ^{g}\!J}^{G_{\yo,r}}\, ^{g}\!\rho\right)^{G_{\yo,\rtm^+}}\neq 0$.
Let $\Xs\in\mathfrak z_{\lieG^{d-1}}$ be a good element representing $\phi_{d-1}$. Let $(G_{\yo,r},\phi)$ be a minimal $K$-type appearing in $(\Ind_{G_{\yo,r}\cap \,^g\! J}^{G_{\yo,r}}\,\,^g\!\rho)^{G_{\yo,r^+}}$.
Then, by Lemma \ref{lem: good type}, there are $h\in G_{[\yo]}S_{\xo}$ and $\eta\in \lieG^{d-1}_{(-\rtm)^+}$ such that $\phi$ is represented by $^h(\Xs+\eta)$.  On the other hand, $G_{g\xo,\rtm}\subset\,^g\!J$ and $\,^g\!\rho|G_{gx,r}$ is a self direct sum of $\,^g\!\phi_{d-1}$. Therefore $\phi$ is a $(G_{\yo,r}\cap \,G_{g\xo,r})$-subrepresentation of such a self direct sum. This means that $\phi=\,^g\!\phi_{d-1}$ on $G_{\yo,r}\cap G_{g\xo,\rtm}$. Equivalently, $\left(\,^h\!(\Xs+\eta)+\lieG_{\yo,(-\rtm)^+}\right)\cap \,^g\!(\Xs+\lieG_{\xo,(-\rtm)^+})\neq\emptyset$ in terms of dual cosets. By \cite[Cor. 2.3.5]{Asym1}, this is in turn equivalent to $\,^{hG_{h^{-1}\xo,0^+}}\!(\Xs+\eta+\lieG^{d-1}_{h^{-1}\yo,(-\rtm)^+})\cap \,^{gG_{\xo,0^+}}\!(\Xs+\lieG^{d-1}_{\xo,(-\rtm)^+})\neq\emptyset$.
This implies $h^{-1}g\in G_{h^{-1}\yo,0^+}C_G(\Xs)G_{\xo,0^+}$ by \cite[Lem. 2.3.6]{Asym1}. Hence, $g\in G_{\yo,0^+}hC_G(\Xs)G_{\xo,0^+}$. It follows that $h^{-1}g\in G_{y,0^+}G^{d-1}G_{x,0^+}$. Hence, $g\in G_{[\yo]}S_{\xo}G_{y,0^+}G^{d-1}G_{x,0^+}\subset S_{\xo,\yo} G^{d-1}G_{x,0^+}$.  \qed

\

Observe that $G_zT^-(\urtm)G_z=G_CWT^-(\urtm)WG_C$, and $G_C\subset G_{\xo}$. Combining these with Lemmas \ref{lem: X_0-1} and \ref{lem: X_0-2}, \begin{equation}\label{eq:used-in-4.15}
\cX^\circ\subset (\mathcal S_{\xo,\yo}\,G^{d-1}G_{\xo})\cap (G_{\xo} W T^-(\urtm)W G_{\xo})
\end{equation}
 when 
$\HypB$ and $\HypGT$ are valid.
 We note that $\mathcal S_{\xo,\yo}$ depends only on $(G,G'=G^{d-1},\xo)$ and $y$.

\begin{prop}\label{p:JXGx}
The double coset space $G_{\xo} \backslash \cX^\circ/J$ is finite. More precisely, setting
$L_s:=G^{d-1}_{[\xo]} G_{\xo,s}$, we have
\begin{equation}\label{eq: bound}
|G_{\xo}\backslash \cX^\circ/J|\le \const_{x, G'}\cdot\sharp(W)^2 (\urtm+3)^{\abrank_G} [L_s:J]\cdot\vol_{G/Z_G}(L_s)^{-\frac12},
\end{equation}
for some constant $\const_{x,G'}>0$ (which may be chosen explicitly; see Lemma \ref{lem: first est} below).
\end{prop}


\proof
Note that each $T^-(\urtm)/(C_G(T)_0Z_G)$ is finite where $C_G(T)_0$ is the maximal parahoric subgroup of $C_G(T)$ and $\sharp \left(T^-(\urtm)/(C_G(T)_0Z_G)\right)\le (\urtm+3)^{\abrank_G}$. Hence, $\sharp \left(G_{z}\backslash\cX/(G_{[z]})\right)\le(\urtm+3)^{\abrank_G}$ by \eqref{pfGzX}.
Since $G_{C}N_G(T)G_{C}=G_{\xo}N_G(T)G_{\xo}=G_{z}T^-G_{z}$ and $N_G(T)/(C_G(T)_0Z)=W (T^-/(C_G(T)_0Z))W$, we have
\[\sharp(G_{\xo} \backslash \cX^\circ/G_{[\xo]})\le (\sharp W)^2 (\urtm+3)^{\abrank_G}. \]
  Since $|G_{\xo}\backslash \cX^\circ/J|=\sum_{\xo\in G_{\xo} \backslash \cX^\circ/G_{[\xo]}} \left|G_{\xo}\backslash G_{\xo}g G_{[\xo]}/J\right|$, the proof of \eqref{eq: bound} is completed by the following lemma:

%

\begin{lem}\label{lem: first est} 
If $G_{\xo}gJ\subset \cX^\circ$, then
\[
\left|G_{\xo}\backslash G_{\xo}g G_{[\xo]}/J\right|\le 
\const_{x, G'}\cdot[L_s:J]\,\vol_{G/Z_G}(L_s)^{-\frac12}
\]
where $\const_{x,G'}:=\vol_{G/Z_G}(G_{[x]})^{-1/2}\cdot\sharp\left(Z_GG_{\xo}\backslash \mathcal S_{\xo,\yo}\right) \cdot q^{(2\dim_k(G)+\dim_k(N))}\,\sharp\left(G_{[\xo]}/(Z_GG_{\xo})\right)$.
\end{lem}

\proof
By \eqref{eq:used-in-4.15}, there is a $w\in G^{d-1}$ such that $G_{x} g G_{x}\subset\mathcal S_{\xo,\yo} wG_{\xo}$.
Then
\[\left|G_{\xo}\backslash G_{\xo}g G_{[\xo]}/J\right|
\le\left|G_{\xo}\backslash \mathcal S_{\xo,\yo}w G_{[\xo]}/J\right|
\le [L_s:J]\cdot\sharp\!\left((Z_GG_{\xo}\backslash \mathcal S_{\xo,\yo}\right)\cdot \sharp\!\left(G_{\xo} w G_{[\xo]}/L_s\right) .\]
Let $\bT$ be as before. Let $\bT^{d-1}$ be a maximal and maximally $k$-split torus of $\bG^{d-1}$ such that the $k$-split components $T^0_k$ and $T^{d-1}_k$ of $T^0$ and $T^{d-1}$ respectively satisfy that $T^0_k\subset T^{d-1}_k\subset T$ and  $\xo\in\Apt(\bT^0)\subset\Apt(\bT^{d-1})\subset\Apt(\bT)$. By the Iwahori decomposition of $G^{d-1}$ one may write $w=u_1w_0u_2$ with $u_1,u_2\in G^{d-1}_x$ and $w_0\in N_{G^{d-1}}(T^{d-1})$. Replacing $w$ with $w_0$ if necessary, one may assume that $w\in N_{G^{d-1}}(T^{d-1})$ since this doesn't change $\mathcal S_{\xo,\yo} wG_{\xo}$.

It is enough to show that there is a unipotent subgroup $U$ such that
\[
\left|G_{\xo}\backslash G_{\xo}w G_{[\xo]}/L_s\right|\le q^{2\dim(G)}\cdot\sharp\left(G_{[\xo]}/(Z_GG_{\xo})\right)\cdot[(U\cap G_{[\xo]}):(U\cap L_s)].
\]
Indeed if the inequality is true, since $G_{[\xo]}\cap U=G_{\xo}\cap U$, Lemma \ref{lem: inequality} applied to $U$ implies that
\[[(U\cap G_{[\xo]}):(U\cap L_s)]\le q^{\dim_k N} [G_{[x]}:L_s]^{-1/2}\le \vol_{G/Z_G}(G_{[x]})^{-1/2} q^{\dim_k N} \vol_{G/Z_G}(L_s)^{-1/2},\]
ending the proof. (As $x$ is fixed, $\const_0=\vol_{G/Z_G}(G_{[x]})^{-1/2}$ depends only on the Haar measure of $G$.)

It remains to find a desired $U$. Let $\bM=C_{\bG}(\bT^{d-1}_k)$ (resp. $\bM^{d-1}:=C_{\bG^{d-1}}(\bT_k^{d-1})$) be the minimal Levi subgroup of $\bG$ (resp. $\bG^{d-1}$) containing $\bT^{d-1}$. Write $w:=w_0t$ with $w_0$ in a maximal parahoric subgroup of $G^{d-1}$ containing $G^{d-1}_x$ and $t\in T^{d-1}_0$.
Let $\bL$, $\bU$ and $\overline{\bU}$  associated to $t$ as in \cite{Del76} (so that our $\bL$, $\bU$ and $\overline{\bU}$ are his $M_g$, $U^+_g$, $U^-_g$ for $g=t$). That is,
$\bL=\{\ell\in G\mid \,\{\,^{t^n}\! \ell\}_{n\in\bbZ}\textrm{ is bounded}\}$, $\bU=\{u\in G\mid \,^{t^n}\! u\ra\infty\textrm{ as }n\ra\infty\}$ and $\overline{\bU}=\{\overline u\in G\mid \,^{t^n}\!{\overline u}\ra 1\textrm{ as }n\ra\infty\}$. Then, $U$ and $\overline U$ are opposite unipotent subgroups with respect to the Levi subgroup $L$. Note that $\bM\subset\bL$. Now, since $G_{\xo,0^+}$ and $G^{d-1}_{\xo,0^+}G_{\xo,s}$ are decomposible with respect to $L, U, \overline U$, we have
\begin{eqnarray}
\left|G_{\xo}\backslash G_{\xo}w G_{[\xo]}/L_s\right|&\le & q^{\dim_k(G)}\sharp\left(G_{[\xo]}/(Z_GG_{\xo})\right)\left|G_{\xo}\backslash G_{\xo}w G_{x,0^+}/(G_{x,0^+}\cap L_s)\right|\nonumber\\
&\le& q^{2\dim_k(G)}\,\sharp\left(G_{[\xo]}/(Z_GG_{\xo})\right)\,[G_{\xo,0^+}\cap U:  (G^{d-1}_{\xo,0^+}G_{\xo,s})\cap U]\nonumber\\
&\le& q^{2\dim_k(G)}\,\sharp\left(G_{[\xo]}/(Z_GG_{\xo})\right)\,[G_{\xo}\cap U:  L_s\cap U].\nonumber
\end{eqnarray} \qed

\subsection{Proof of the main theorems} \label{sub:proof-main-thm}
This subsection is devoted to the proof of Theorem \ref{t:asymptotic-char} and its quantitative version in Theorem \ref{thm:uniform-bound}.
Combining the estimates of the previous subsection, we have the following:
\begin{prop}\label{prop: main} Suppose $\HypB$, $\HypGT$, and $\Hypk$ are valid. Let $\pi=\pi_\Sigma$. Let $\gamma\in G_{0}\cap G_{\rs}$ with $\sd(\gamma)\le\frac{\rtm}2$. Then, we have

\begin{equation}\label{eq: main}
\left|\frac{\Theta_\pi(\gamma)}{\deg(\pi)}\right|\le  \const_1 \cdot (\sharp W)^2\cdot q^{\dim(G)+\abrank_G(A_{\gamma,\Sigma}+1)}\cdot D(\gamma)^{-1} \cdot (r+4)^{\abrank_G} \vol_{G/Z}(L_s)^{\frac12},
\end{equation}
where  $\const_1=\max\{\const_{x, G'}\}$ where the maximum runs over the finitely many $G$-orbits of $(x, G')$ and $\const_{x,G'}$ is the constant as in Lemma \ref{lem: first est}.


\end{prop}

\proof


Without loss of generality, we may reduce to cases as in \S\ref{rmk: reductions}. In the following, $\psi:G_{\xo}\rightarrow G_{\xo}$ is the map defined by $\psi(g)=g\gamma g^{-1}$ and for $g\in G_{\xo}$, $H_g:=G_{\xo}\cap \, ^g\!J$. Our starting point is formula \eqref{e:loc-const-refined} computing $\Theta_\pi(\gamma)$.
The summand for each $g$ has a chance to contribute to the trace of $\gamma$ only if $^{g'}\! \gamma\in G_{\xo}\cap \, ^g\!J$ for some $g'\in G_{\xo}$. Consider $\Ind^{G_{x}}_{G_{x}\cap\,^g\!J} {}^g \rho$ in the summand of $g$ as the direct sum of the spaces of functions supported on exactly one left $G_{\xo}\cap \, ^g\!J$-coset in $G_{\xo}$. The element $\gamma$ permutes the spaces by translating the functions by $\gamma$ on the right. So it is easy to see that each space may contribute to the trace of $\gamma$ on $V_\pi^{G_{\yo,r^+}}$ only if the supporting $G_{\xo}\cap \, ^g\!J$-coset is fixed by $\gamma$.
Hence we have


\begin{align*}
\left|\frac{\Theta_\pi(\gamma)}{\deg(\pi)}\right|
&=\frac{\vol_{G/Z_G}(J)}{\dim \rho}\left|\mathrm{Tr}\left(\pi(\gamma)|V_\pi^{G_{\yo,r^+}}\right)\right| \\
&\le\frac{\vol_{G/Z_G}(J)}{\dim \rho}\sum_{g\in G_{\xo}\backslash\cX^\circ/J\atop\mathrm{s.t.}~{}^{g'}\!\gamma\in H_g}
\# \Fix(\gamma|(G_{\xo}\cap \, ^g\!J)\bs G_{\xo})\cdot \dim \rho\\
&=\frac{\vol_{G/Z_G}(J)}{\dim \rho}\sum
_{g\in G_{\xo}\backslash\cX^\circ/J\atop\mathrm{s.t.}~{}^{g'}\!\gamma\in H_g}
{[\psi^{-1}(H_g):H_g]}\dim\rho \\
&={\vol_{G/Z_G}(J)}\sum_{g\in G_{\xo}\backslash\cX^\circ/J\atop\mathrm{s.t.}~{}^{g'}\!\gamma\in H_g}
{[\psi^{-1}(H_g):H_g]} \\
&\le \const_1\cdot q^{\dim(G)+\abrank_G(\aconst_{\gamma,\Sigma}+1)}\cdot D(\gamma)^{-1} \cdot(\sharp W)^2\, (\urtm+3)^{\abrank_G}\, [L_s:J]\, \vol_{G/Z_G}(L_s)^{-\frac12}\vol_{G/Z_G}(J)\\
&=\const_1 \cdot q^{\dim(G)+\abrank_G(\aconst_{\gamma,\Sigma}+1)}\cdot D(\gamma)^{-1} \cdot(\sharp W)^2\, (r+4)^{\abrank_G}\vol_{G/Z_G}(L_s)^{1/2}.
\end{align*}
The $\sum$ above runs over $ g\in G_{\xo}\backslash\cX^\circ/J$ such that $^{g'}\!\gamma\in H_g$ for some $g'\in G_x$.
The second last inequality follows from Corollary \ref{cor:conj-inv-estimate} and Proposition \ref{p:JXGx}.  The last equality follows from
$[L_s:J]\vol_{G/Z_G}(J)=\vol_{G/Z_G}(L_s)$.
\qed

\begin{proof}[Proof of Theorem \ref{t:asymptotic-char}]
Let $\pi_i=\pi_{\Sigma_i}$, $i=1,2,\cdots$ be a sequence of supercuspidal representations in $\Irr^{\Yu}(G)$ with $\deg(\pi_i)\rightarrow\infty$.
Recall from \S\ref{rmk: reductions}, we may assume $x_{\Sigma_i}\in\Sigma_i$ is in a fixed $G$-orbit. Since $\deg(\pi_i)\rightarrow\infty$, we have $r_{\Sigma_i}\rightarrow\infty$ and $r_{\Sigma_i}>2\sd(\gamma)$ for almost all $\pi_i$. Hence, (\ref{eq: main}) holds for $\pi$ with $i$ large enough. Note that in (\ref{eq: main}), when $h_G>1$, only $(r_{\Sigma_i}+4)^{\abrank_G} \cdot q^{\abrank_G s}\vol_{G/Z_G}(L_{s_{\Sigma_i}})^{\frac12}$ varies as $\pi_i$ varies with $i$ large enough. It suffices to show that this quantity approaches zero as $r_{\Sigma_i}\ra\infty$. Since the term $(r+4)^{\abrank_G}$ has a polynomial growth in $r$ while $q^{\abrank_Gs}\vol_{G/Z}(L_{s})$ decays exponentially as $s=r/2$ tends to infinity by Lemma \ref{lem:bound-vol(L_s)}. The case $h_G=1$ is similar, hence we are done.
\end{proof}

\begin{rem}
It is also interesting to discuss the role of the subgroup $G_{\yo,r^+}$ in our proof. This subgroup appears in Lemma~\ref{lem: loc const} from the local constancy of the character $\Theta_\pi$. Of course for the purpose of that lemma any open subgroup $K\subset G_{\yo,r^+}$ would work. However, from the fact that $G_{\yo,\rtm}$-representations in $V_\pi^{G_{\yo,\rtm^+}}$ are minimal $K$-types, we acquire another description of $\cX^\circ$ as in Lemma \ref{lem: X_0-2}, which is again used to get an estimate in Lemma \ref{lem: first est}.
\end{rem}

  In the remainder of this section we upgrade Theorem \ref{t:asymptotic-char} to a uniform quantitative statement. From here on we may and will normalize the Haar measure on $G/Z$ such that $\vol(G_{[x]}/Z)=1$. This is harmless because there are finitely many conjugacy classes of $(\vec\bG,\xo)$ as explained in \S\ref{rmk: reductions}.

\begin{lem}\label{lem:bound-vol(L_s)} \ 
\begin{enumerate}
\item[(i)] If $h_G>1$,
 there exists a constant $\kappa>0$ such that for all $\pi=\pi_\Sigma\in \Irr^{\Yu}(G)$,
  \[q^{\abrank_G s}\vol(L_s)^{1/2} \le q^{\dim G}\cdot \deg(\pi)^{-\kappa}.\]
\item[(ii)] If $h_G=1$,
 there exists a constant $\kappa>0$ such that for all $\pi=\pi_\Sigma\in \Irr^{\Yu}(G)$,
  \[\vol(L_s)^{1/2} \le q^{\dim G}\cdot \deg(\pi)^{-\kappa}.\]
 \end{enumerate}
\end{lem}

\begin{proof}
(i) Observe that
  \begin{eqnarray}
  \vol(J/Z)^{-1}&=&[G_{[x]}:J]\le [ZG_x:ZG_{x,s}]=[G_x:(Z\cap G_x)G_{x,s}]\le q^{(\dim G-\dim Z)s},\nonumber\\
  \vol(L_s/Z)^{-1}&=&[G_{[x]}:L_s]\ge [G_{[x]}:G'_{[x]}G_{x,s}]\ge [G_{x,0^+}:G'_{x,0^+} G_{x,s}]\ge q^{(\dim G-\dim G')(s-1)},\nonumber
  \end{eqnarray}
with $G'=G^{d-1}$. Recall that $\deg(\pi)\le q^{\dim G}\vol(J/Z)^{-1}$. Take
\[
\kappa:=\min_{G'\subsetneq G}\frac{\dim G-\dim G'-2\abrank_G}{2(\dim G-\dim Z)},
\]
where $G'$ runs over the set of proper tamely ramified twisted Levi subgroups of $G$. If we know $\kappa>0$ then the lemma follows from the following chain of inequalities:
\begin{eqnarray}
\deg(\pi)^{\kappa}&\le& q^{\kappa \dim G}q^{(\dim G-\dim G'-2\abrank_G)s/2}\le q^{\dim G/2} q^{(\dim G-\dim G'-2\abrank_G)s/2}\nonumber\\
&\le& q^{\dim G}\cdot q^{-\abrank_Gs}\cdot\vol(L_s/Z)^{-1/2}.\nonumber
\end{eqnarray}
It remains to show that $\kappa>0$.
Since $\dim\bG=\dim_k G$, it is enough to show that $\dim\bG-\dim\bM-2\abrank_G>0$ when $\bM$ is a proper Levi subgroup which arises in a supercuspidal datum. This can be seen as follows. If $\bG$ of type other than $A$, the inequality holds for any proper Levi subgroup $\bM$. If $\bG$ is of type $A_n$, then $\kappa>0$ unless $\bM$ is of type $A_{n-1}$. However such a Levi subgroup does not arise as part of supercuspidal datum when $n\ge 2$, and the assumption $h_G>1$ excludes the case $n=1$.

(ii) In this case, we can take $\kappa= \min_{G'\subsetneq G}\frac{\dim G-\dim G'}{2(\dim G-\dim Z)}$. It is clear that $\kappa >0$ and the rest of the proof works as in~(i).
\end{proof}

  Since $J\subset L_s$, the above proof implies the lower bound $\vol(J/Z)^{-1}\ge q^{(\dim G-\dim G')(s-1)}$. Combined with Lemma \ref{l:bound-on-d}, this yields $\deg(\pi)\ge q^{(\dim G-\dim G')(s-1)}$. The following theorem is an improvement of Theorem \ref{t:asymptotic-char}.

\begin{thm}\label{thm:uniform-bound} Assume hypotheses $\HypB$, $\HypGT$ and $\Hypk$. There exist constants $A,\kappa,C>0$ depending only on $G$ such that the following holds.
For every $\gamma\in G_{0}\cap G_{\rs}$ and $\pi\in\Irr^{\Yu}(G)$ such that $\sd(\gamma)\le r/ 2$,
  \begin{equation}\label{eqn: uniform}
  D(\gamma)^A |\Theta_\pi(\gamma)|\le C\cdot \deg(\pi)^{1-\kappa}.
  \end{equation}
\end{thm}

\begin{proof} Let $\pi:=\pi_\Sigma\in \Irr^{\Yu}(G)$ and recall that $r=r_\Sigma$ is the depth of $\pi$.
Since $\gamma\in G_{0}$ we have $\nu(1-\alpha(\gamma))\ge 0$ for all $\alpha\in \Phi(\bT^\gamma)$. In view of Definition \ref{d:singular-depth},
   we have \[D(\gamma)=\prod_{\alpha\in \Phi(\bT^\gamma)} |1-\alpha(\gamma)|\le q^{-\sd(\gamma)}\le 1.\]
  Proposition \ref{prop: main} yields a bound of the form
  \[D(\gamma)\left|\frac{\Theta_\pi(\gamma)}{\deg(\pi)}\right|\le C\cdot q^{\abrank_G A_{\gamma,\Sigma}}\vol(L_s)^{1/2},\]
  where $C\in \R_{>0}$ is a constant depending only on $G$.
  Consider the case that $h_G>1$. Recall that $A_{\gamma,\Sigma}=h_G\sd(\gamma)+s$. Let us take $A:=\abrank_Gh_G+1$. Then
  \[D(\gamma)^A \left|\frac{\Theta_\pi(\gamma)}{\deg(\pi)}\right|\le q^{-\abrank_Gh_G\sd(\gamma)}D(\gamma)\left|\frac{\Theta_\pi(\gamma)}{\deg(\pi)}\right|
  \le C\cdot q^{\abrank_G\cdot s}\vol(L_s)^{1/2}.\]
Then, by  Lemma \ref{lem:bound-vol(L_s)}.(i), we have
\[
D(\gamma)^A |\Theta_\pi(\gamma)|\le C_0\cdot \deg(\pi)^{1-\kappa}
\]
  with $C_0:= C q^{\dim G}$ and the same $\kappa$ as in that lemma, completing the proof when $h_G>1$.
  
  In the remaining case $h_G=1$, we have $A_{\gamma,\Sigma}=\sd(\gamma)$. Take $A:=r_G+1$. Then
   \[D(\gamma)^A \left|\frac{\Theta_\pi(\gamma)}{\deg(\pi)}\right|\le q^{-\abrank_G \sd(\gamma)}D(\gamma)\left|\frac{\Theta_\pi(\gamma)}{\deg(\pi)}\right|
  \le C\cdot \vol(L_s)^{1/2}.\]
  The proof is finished by applying Lemma \ref{lem:bound-vol(L_s)}.(ii) and taking $C_0= C q^{\dim G}$ again.
  
\end{proof}

\subsection{On the assumption $\sd(\gamma)\le r/ 2$.}\label{sub:assumption}
Theorem~\ref{thm:uniform-bound} above remains valid without the assumption $\sd(\gamma) \le r/ 2$.
Indeed the uniform estimate~\eqref{eqn: uniform} holds in the range $r < 2 \sd(\gamma)$ by a different argument that we now explain.\footnote{A priori we are proving the bound~\eqref{eqn: uniform} in two disjoint regions with two different values of $(A,\kappa)$; call them $(A_1,\kappa_1)$ and $(A_2,\kappa_2)$. When we say that Theorem~\ref{thm:uniform-bound} is valid without the assumption $\sd(\gamma) \le r/ 2$, it means that there's a single choice of $(A,\kappa)$ that works in both regions. This is immediate because $\gamma \in G_{0^+}$, in which case it follows that $D(\gamma)\ll 1$. So it enough to take $A=\max(A_1,A_2)$ and $\kappa=\min(\kappa_1,\kappa_2)$, possibly at the expense of increasing the constant $C$ in~\eqref{eqn: uniform}.}

It suffices to produce a polynomial upper-bound on the trace character  $|\Theta_\pi(\gamma)|$ in the range $r < 2 \sd(\gamma)$.
By \cite[Cor 12.9]{AK07} and \cite{Meyer-Solleveld:growth} (as explained in the proof of Lemma \ref{lem: loc const}) the character $\Theta_\pi$ is constant on $\gamma G_{y,t}$ if $t=2\sd(\gamma) +1$.
The analogue of~\eqref{eq: loc const} holds, hence $|\Theta_\pi(\gamma)| \le \dim V_\pi^{G_{y,t}}$. It is sufficient to show under the same hypotheses $\HypB$, $\HypGT$ and $\Hypk$ as above that for all $\pi\in\Irr^{\Yu}(G)$, $y\in \Bd(G)$, and $t\ge 1$,
\begin{equation}\label{dim-bound}
\dim V_\pi^{G_{y,t}} \le q^{Bt}
\end{equation}
for some constant $B>0$ depending only on $G$. Indeed $D(\gamma)\le q^{-\sd(\gamma)}$ and $2s =r < t = 2\sd(\gamma)+1$ so~\eqref{eqn: uniform} follows with $A=2B$ and $\kappa=1$.

Similarly as in Definition~4.8 we introduce the subset $\cX^\circ \subset G$.
Mackey's decomposition implies that $ \dim V_\pi^{G_{y,t}} \le \dim \rho \cdot [G_x:G_{y,t}] \cdot |G_x \backslash \cX^\circ / J |$.
The dimension $\dim \rho$ is uniformly bounded by Lemma~\ref{l:bound-on-d}, the term $[G_x:G_{y,t}]$ is polynomially bounded (see Lemma~\ref{lem:bound-vol(L_s)}), as well as the third term by Proposition~\ref{p:JXGx}. This establishes~\eqref{dim-bound}.

\section{Miscellanies}
\subsection{Trace characters versus orbital integrals of matrix coefficients}
The bulk of the proof above was to establish a power saving as $\deg(\pi)\to \infty$. In fact we can develop two distinct approaches, the first of which is taken in this paper.
\begin{enumerate}[(i)]
\item We have handled the trace characters $\Theta_\pi(g)$ using first their local constancy~\cite{SS97}. The proof then was by a uniform estimate on the irreducible factors of the restriction of $\cind_J^G \rho$ to a suitable subgroup.
\item The other approach developed in~\cite{KST} is via the orbital integral $O_{\gamma}(\phi_\pi)$ of a matrix coefficient $\phi_\pi$ which can be written explicitly from Yu's construction. The proof is via a careful analysis of the conjugation by $\gamma$ on $J$ and uses notably a recent general decomposition theorem of Adler--Spice~\cite{Adler-Spice:expansion}.
\end{enumerate}

For $\gamma$ regular \emph{and} elliptic semisimple the orbital integral $O_{\gamma}(\phi_\pi)$  and the trace character $\Theta_\pi(\gamma)$ coincide as follows from the local trace formula of Arthur; in this case our approaches (i) and (ii) produce similar estimates. The approach (i), where $\gamma$ is regular, is well-suited to establish our proposed conjecture, while the approach (ii), where $\gamma$ is elliptic (but not regular), is well-suited for application of the trace formula. Indeed the goal of \cite{KST} is to establish properties of families of automorphic representations, similarly to~\cite{ST11cf}, as we prescribe varying supercuspidal representations at a given finite set of primes.

\subsection{Analogues for real groups}\label{sub:real-group-char}
We would like to see the implication of Harish-Chandra's work on the analogue of Conjecture \ref{c:asymptotic-char} for real groups. Only in this subsection, let $G$ be a connected reductive group over $\R$. Write $A_G$ for the maximal split torus in the center of $G$ and put $A_{G,\infty}:=A_G(\R)^0$. Then $G(\R)$ has discrete series if and only if $G$ contains an elliptic maximal torus $T$ over $\R$, namely a maximal torus $T$ such that $T(\R)/A_{G,\infty}$ is compact. Fix a choice of $T$ and a maximal compact subgroup $K\subset G(\R)$ such that $T\subset K A_{G,\infty}$. Let $W_\R$ denote the relative Weyl group for $T(\R)$ in $G(\R)$. Write $\mathfrak{t}:=\Lie T(\R)$ and $\mathfrak{t}^*$ for its linear dual. Set $q(G):=\frac12\dim_\R (G(\R)/K)\in \Z$. Let $\pi$ be an (irreducible) discrete series of $G(\R)$ whose central character is unitary, and denote by $\lambda_\pi\in i\mathfrak{t}^*$ its infinitesimal character. Let $\gamma$ be a regular element of $T(\R)$, which is uniquely written as $\gamma=z\exp H$ for $z\in K\cap Z(G(\R))$ and $H\in \Lie T(\R)$.

\begin{prop}
The real group analogue of Conjecture \ref{c:asymptotic-char} is verified for elliptic regular elements $\gamma$ and discrete series with unitary central characters.
\end{prop}

\begin{proof}
  Harish-Chandra's character formula for discrete series on elliptic maximal tori implies that (as usual $D(\gamma)$ is the Weyl discriminant)
  \[ D(\gamma)^{1/2}\Theta_\pi(\gamma)=(-1)^{q(G)}\sum_{w\in W_\R} \sgn(w)e^{\lambda_\pi(H)}.\]
  Hence $D(\gamma)^{1/2}|\Theta_\pi(\gamma)|\le |W_\R|$.
\end{proof}

   Note that we have a much stronger version for part (ii) of the conjecture, allowing $\epsilon=1$. To verify part (i) when $\gamma$ is contained in a non-elliptic maximal torus, one can argue similarly by using the character formula due to Martens \cite{Mar75} as far as holomorphic discrete series are concerned. A general approach would be to use a similar character formula as above, which exists but comes with a subtle coefficient in each summand which depends on $w$ and $\pi$. The coefficients can be analyzed in two steps: firstly one studies the analogous coefficients for stable discrete series characters (as studied by Herb; also see \cite[p.273]{Art89}), and secondly relates the character of a single discrete series of $G(\R)$ to the stable discrete series characters on endoscopic groups of $G(\R)$ following the idea of Langlands and Shelstad. For instance this has been done in \cite{Her83} (also see \cite[p.273]{Art89} for the first step). On the other hand, Herb has another approach avoiding endoscopy in \cite{Her98}.
   We do not pursue either approach further in this paper as it would take us too far afield.

\subsection{Analogues for finite groups}\label{sub:finite-gp}
Let $G$ be a finite group of Lie type over a finite field with $q\ge 5$ elements. Gluck~\cite{Gluck:character-sharp} has shown that if $\pi$ is a nontrivial irreducible representation of $G$ and $\gamma$ is a noncentral element, then the trace character satisfies
\[
|\chi_\pi(\gamma)| \le \frac{\dim(\pi)}{\sqrt{q}-1}.
\]
The bound has interesting applications, see e.g. \cite{Gluck:random-walks,Liebeck-Shalev:fuchsian,Liebeck-Shalev:diameters,Ore-conj}.

\subsection{Open questions}\label{sub:dep-on-gamma}
 In this subsection we raise the question of the possible upper-bounds on $\Theta_\pi(\gamma)$ in terms of both $\pi$ and $\gamma$. One may ask about the sharpest possible bound.
 Our main result was a bound of the form (Theorem \ref{thm:uniform-bound})
\begin{equation}\label{uniform-bound}
 D(\gamma)^A |\Theta_\pi(\gamma)|  \le C \fdeg(\pi)^{\kappa},
\end{equation}
where $C$ is independent of $\gamma\in G_{0+}\cap G_{\rs}$ and $\pi\in \Irr^{\Yu}(G)$.
Slightly more generally we fix a bounded subset $\mathcal{B}\subset G$ and assume in the following that $\gamma\in \mathcal{B}\cap G_{\rs}$.

The most optimistic bound would be that
\begin{equation}\label{optimistic}
 D(\gamma)^{\frac12} |\Theta_\pi(\gamma)| \stackrel{?}{\le} C,
\end{equation}
where $\gamma\in \mathcal{B}$ and $C$ depends only on $\mathcal{B}$. In the appendix we shall verify that the estimate~\eqref{optimistic} is valid for the group $G=\SL_2(k)$. However the analogue of this bound already doesn't hold in higher rank when varying the residue characteristic and $\pi$ is a Steinberg representation, as we explain in~\cite{KST}. It would be interesting to investigate all the counterexamples to~\eqref{optimistic} in general and find in which cases it holds.

There is a wide range of possibilities between~\eqref{uniform-bound} and~\eqref{optimistic}. The exact asymptotic of $|\Theta_\pi(\gamma)|$ lies somewhere in between, and based on Harish-Chandra regularity theorem and our work in~\cite{KST} it seems plausible that the exact bound should be $A=\frac12$ and $\kappa$ slightly below $1$ depending on~$G$.

\appendix

\section{The Sally-Shalika character formula}
We study the Sally--Shalika formula~\cite{Sally-Shalika:characters} for characters of admissible representations of $G=\SL_2(k)$, where $k$ is a $p$-adic field.
Our goal is to establish a bound of the form \eqref{optimistic} for the character $\Theta_\pi(\gamma)$ that is completely uniform in $\pi$. The explicit calculation of character values is crucial for this. It would be interesting to investigate where such a result can hold in general.

Let $ Z=\set{\pm 1}$ be the center of $G$.
We follow mostly the notation and convention from~\cite{ADSS:supercuspidal-characters}.
We assume throughout that $p\ge 2e+3$ where $e$ is the absolute ramification degree of $k$. We let $\theta\in \set{\varepsilon,\varepsilon\varpi,\varpi}$, with $\varpi$ is a uniformizer of $\cO_k$ and $\varepsilon$ is any fixed element of $\cO^\times_k\bs (\cO^\times_k)^2$. Let $k_\theta:=k(\sqrt{\theta})$ and $k^1_\theta\subset k^\times_\theta$ be the subgroup of elements of norm one. We extend the valuation from $k^\times$ to $k_\theta^\times$. Note that there is a unique non-trivial quadratic character $\varphi_\varepsilon:k^1_\varepsilon\to \set{\pm 1}$.

Inside $G$ we let $T^\theta \simeq k^1_\theta$ be the associated maximal elliptic tori given by the matrices $\mdede{a}{b}{b\theta}{a}$ where $a+b\theta\in  k^1_\theta$. As $\theta$ ranges in $\set{\varepsilon,\varepsilon\varpi,\varpi}$ this describes the stable conjugacy classes of elliptic tori (abstractly $k_\theta$ is the splitting field of $T^\theta$). There is a finer classification of $G$-conjugacy classes: there are two unramified conjugacy classes of unramified elliptic tori, denoted $T^{\varepsilon}=T^{\varepsilon,1}$ and $T^{\varepsilon,\varpi}$ while for ramified elliptic tori, the answer depends on whether $-1$ is a square in the residue field. If $-1$ is not a square then besides $T^{\varpi}=T^{\varpi,1}$ and $T^{\varepsilon\varpi}=T^{\varepsilon\varpi,1}$ there are two  additional $G$-conjugacy classes denoted $T^{\varpi,\varepsilon}$ and $T^{\varepsilon\varpi,\varepsilon}$.

The torus filtration is as described in~\cite[\S3.2]{ADSS:supercuspidal-characters}, namely an element $1+x\in k^1_\theta\simeq T^\theta$ with $v(x)>0$ has depth equal to $v(x)$. In particular we have $D(\gamma)=q^{-2d_+(\gamma)}$ for all regular semisimple $\gamma\in G$ where $d_+(\gamma):=\max\limits_{z\in Z} d(z\gamma)$ is the maximal depth.

Every supercuspidal representation of $G$ is of the form $\pi=\pi^{\pm}(T,\varphi)$ where $(T,\varphi,\pm)$ is a supercuspidal parameter. Here $T$ is an elliptic tori up to $G$-conjugation and $\varphi$ is a quasi-character of $T$. The depth of $\pi$ is equal to the depth of $\varphi$ which is the smallest $r\ge 0$ such that $\varphi$ is trivial on $T_{r^+}$.

Let $dg$ be the Haar measure on $G/Z(G)$ is as in~\cite[\S6]{ADSS:supercuspidal-characters}, thus $\Mvol(\SL(2,\cO_k))=\frac{q^2-1}{q^{\frac12}}$. The formal degree is by construction $\fdeg(\pi)=\frac{\dim(\rho)}{\Mvol(J)}$. By a theorem of Harish-Chandra $\fdeg(\pi)$ is proportional to the constant term $c_0(\pi)$ in the expansion of $\Theta_\pi$ near the identity. Here we find $\fdeg(\pi)=c\cdot c_0(\pi)$ where $c:=-\frac{2q^{\frac12}}{q+1}$.

The Sally--Shalika formula is an exact formula for the character $\Theta_\pi(\gamma)$ for any regular noncentral semisimple element $\gamma\in G$. Here we shall give a direct consequence tailored to our purpose of studying of the asymptotic behavior of characters.
\begin{prop}\label{p:sl2-formula}
If $\pi=\pi(T^\varepsilon,\varphi)$ has depth $r$ then the following holds:
\[
D(\gamma)^{\frac12}
\abs{\Theta_\pi(\gamma)}=
\begin{cases}
\abs{\varphi(\gamma)+\varphi(\gamma^{-1})},
& \gamma\in T^\varepsilon\bs
ZT^\varepsilon_{r^+}\\
1 \pm \fdeg(\pi)  D(\gamma)^{\frac12},
& \gamma\in T^{\varepsilon,\eta}_{r^+},\ \eta\in \set{1,\varpi}\\
1- \fdeg(\pi)  D(\gamma)^{\frac12},
& \gamma\in A_{r^+}\\
\fdeg(\pi)  D(\gamma)^{\frac12},
& \text{o/w if $\gamma\in G_{r^+}$}.
\end{cases}
\]
The character vanishes in the other cases, namely if $\gamma\not\in G_{r^+}\cup T^\varepsilon$.
The formal degree is $\fdeg(\pi)=q^r$.
\end{prop}

\begin{proof}
This is~\cite[\S14]{ADSS:supercuspidal-characters}. Note that in their notation the quasi-character $\varphi$ is denoted $\psi$ there; the additive character $\Omega_k$ is denoted $\Lambda$ there.

Since the Gauss sum $H(\Lambda',k_\varepsilon)$ is unramified, we have that it is equal to $(-1)^{r+1}$ according to~\cite[Lemma 4.2]{ADSS:supercuspidal-characters}.

The assertion on the formal degree is~\cite[Remark 10.16]{ADSS:supercuspidal-characters} since $c_0(\pi)=-q^r$.
\end{proof}

\begin{prop}\label{p:sl2ramified}
If $\pi=\pi(\varpi,\varphi)$ has depth $r$ then the following holds:
\[
D(\gamma)^{\frac12}
\abs{\Theta_\pi(\gamma)}\le
\begin{cases}
2,
& \gamma\in T^\theta\bs
ZT^\theta_{r}\\
1.5,
& \gamma\in T^{\varpi,\eta}_{r}\bs T^{\varpi,\eta}_{r^+},\ \eta\in \set{1,\varpi}\\
1,
& \gamma\in T^{\varepsilon\varpi,\eta}_{r}\bs T^{\varepsilon\varpi,\eta}_{r^+},\ \eta\in \set{1,\varpi}\\
1+\fdeg(\pi)  D(\gamma)^{\frac12},
& \gamma\in T^{\varpi,\eta}_{r^+} \cup A_{r^+}\\
\fdeg(\pi)  D(\gamma)^{\frac12},
& \text{o/w if $\gamma\in G_{r^+}$}.
\end{cases}
\]
The character vanishes in the other cases, namely if $\gamma\not\in G_{r^+}\cup T^\theta$.
The formal degree is $\fdeg(\pi)=\frac12(q+1)q^{r-\frac12}$.
\end{prop}
\begin{proof}
Again  this is~\cite[\S14]{ADSS:supercuspidal-characters} where it is shown that $c_0(\pi)=-\frac12(q+1)q^{r-\frac12}$.

The ramified Gauss sum $H(\Lambda',k_\varpi)$ is a fourth root of unity according to~\cite[Lemma 4.2]{ADSS:supercuspidal-characters}.
In the second case we have the inequality $\le 1 + \abs{A}$ where the exponential sum is
\[
A := \frac{1}{2\sqrt{q}}
\sum_{\substack{x\in (k^1_\varpi)_{r:r^+}\\
x\neq  \gamma^{\pm 1}
}}
\sgn_\varpi(\tr(\gamma-x))
\varphi(x).
\]
Here $k^1_\theta\subset k_\theta^\times$ is the subgroup of elements of norm $1$, and $ (k^1_\varpi)_{r:r^+}$ denotes~\cite[\S5.1]{ADSS:supercuspidal-characters} the quotient group $(k^1_\varpi)_r / (k^1_\varpi)_{r^+}$. This is an additive group that can be described by writing explicitly $x=1+\alpha^{2r}X$ where $X\in \cO/(\varpi)$. We have $\tr(x)=2+\tr(\alpha^{2r})X$, and similarly we shall write $\gamma=1+\alpha^{2r}Y$.

Since $\sgn_\varpi$ is the quadratic character attached to $\varpi$, we are left with $\chi(X-Y)$ where $\chi$ is the Legendre symbol on $\cO/(\varpi)$.
The character $\varphi$ has conductor $r$, thus $X\mapsto \varphi(1+\alpha^{2r}X)$ is a non-trivial additive character. Finally the exponential sum $A$ is a unit times a Gauss sum, thus $\abs{A}=\frac12$.

In the third case the character is equal to an exponential sum which can be handled similarly.
\end{proof}

We finally consider the remaining four \Lquote{exceptional} supercuspidal representations. They all have depth zero.
\begin{prop}\label{p:sl2exceptional}
Suppose that $\pi$ is an exceptional supercuspidal representation induced from $T^\varepsilon$. Then the following holds:
\[
2 D(\gamma)^{\frac12}
\abs{\Theta_\pi(\gamma)}\le
1 + D(\gamma)^{\frac12},
\quad \gamma\in T^\varepsilon\bs ZT^\varepsilon_{0^+} \cup A_{0^+} \cup T_{0^+},
\]
where $T$ is any of the elliptic tori, and the character vanishes otherwise. The formal degree is $\fdeg(\pi)=\frac12$.
If $\pi$ is induced from $T^{\varepsilon,\varpi}$, the same formula holds with $T^\varepsilon$ replaced by $T^{\varepsilon,\varpi}$.
\end{prop}

\begin{rem}
The behavior $D(\gamma)\to \infty$ is qualitatively different than for the other \Lquote{ordinary} supercuspidals.
\end{rem}

\begin{proof}
This follows from~\cite[\S9, \S15]{ADSS:supercuspidal-characters}. The passage from $T^\varepsilon$ to $T^{\varepsilon,\varpi}$ is explained in \cite[Rem.~9.8]{ADSS:supercuspidal-characters}.
\end{proof}

\begin{cor}\label{cor:sl2-bound} For all supercuspidal representations $\pi$ of $\SL(2,k)$ and all regular semisimple $\gamma$, the following holds:
\[
  D(\gamma)^{\frac12}
\abs{\Theta_\pi(\gamma)}\le
  2 + D(\gamma)^{\frac12}.
\]
\end{cor}
\begin{proof}
This follows by combining the Propositions~\ref{p:sl2-formula}, \ref{p:sl2ramified} and \ref{p:sl2exceptional}.
In the last three cases of Proposition~\ref{p:sl2-formula} we need to observe that $\gamma\in G_{r^+}$ which is equivalent to $d(\gamma) > r$. This implies $d_+(\gamma)>r$ and thus $D(\gamma)< q^{-2r}$. Therefore $\fdeg(\pi)D(\gamma)^{\frac12}<1$.

Similarly in the last two cases of Proposition~\ref{p:sl2ramified} we have that $\gamma\in G_{r^+}$ and in view of the normalization of the valuation this implies that $D(\gamma)<q^{-2r-1}$. Therefore $\fdeg(\pi)D(\gamma)^{\frac12}\le \frac12$ which concludes the claim.
\end{proof}

\bibliographystyle{abbrv}

\def\cprime{$'$}\def\cprime{$'$}\def\cprime{$'$}\def\cprime{$'$}\def\cprime{$'$}

\end{document}